\documentclass{birkjour}
\usepackage{amssymb}
\usepackage{relsize}
\usepackage{mathrsfs}
\newtheorem{theorem}{Theorem}[section]
\newtheorem{lemma}[theorem]{Lemma}
\newtheorem{proposition}[theorem]{Proposition}

\newtheorem{corollary}[theorem]{Corollary}
\newtheorem{definition}[theorem]{Definition}

\newtheorem{remark}[theorem]{Remark}

\renewcommand{\theequation}{\arabic{section}.\arabic{equation}}

\newcommand{\Tr}{\mathop{\mathrm{Tr}}}
\renewcommand{\d}{\/\mathrm{d}\/}

\def\w{\textbf{W}^{\varepsilon}_{{\theta}^{\varepsilon}}}

\def\e{\varepsilon}

\def\t{t\wedge\tau_N}
\def\T{T\wedge\tau_N}
\def\x{\mathbf{x}}

\def\y{\textbf{y}}
\def\z{\textbf{z}}
\def\v{\mathbf{v}}
\def\w{\mathbf{w}}
\def\V{\mathbf{V}}
\def\u{\mathbf{u}}
\def\H{\mathbf{H}}
\def\n{\mathbf{n}}

\renewcommand{\d}{\/\mathrm{d}\/}

\begin{document}

\title[Stochastic Navier-Stokes Equations in Unbounded Channel Domains] {Stochastic Navier-Stokes Equations in Unbounded Channel Domains}

\author[Utpal Manna]{Utpal Manna}

\address{%
School of Mathematics\\
Indian Institute of Science Education and Research (IISER) Thiruvananthapuram\\
Thiruvananthapuram 695016\\
Kerala, INDIA}

\email{manna.utpal@iisertvm.ac.in}

\author[Manil T. Mohan]{Manil T. Mohan}

\address{%
School of Mathematics\\
Indian Institute of Science Education and Research (IISER) Thiruvananthapuram\\
Thiruvananthapuram 695016\\
Kerala, INDIA}

\email{manil@iisertvm.ac.in}

\author[Sivaguru S. Sritharan]{Sivaguru S. Sritharan}
\address{%
 Naval Postgraduate School\\
Monterey\\
 USA}

\email{sssritha@nps.edu}

\subjclass{Primary 76D05; Secondary 60H15, 76D03, 76D06}

\keywords{stochastic Navier-Stokes equations, viscous flow in
channels, path-wise strong solutions}

\begin{abstract}
In this paper we prove the existence and uniqueness of path-wise
strong solution to stochastic viscous flow in unbounded channels
with multiple outlets using local monotonicity arguments. We devise
a construction for solvability using a stochastic basic vector
field.
\end{abstract}

\maketitle\setcounter{equation}{0}

\maketitle

\section{Introduction} This paper concerns with stochastic
fluid dynamics in unbounded channel domains with noncompact
boundaries generalizing the deterministic results in Sritharan
\cite{SS4}. Mathematical theory of viscous incompressible flow
through unbounded channel has many applications such as hydraulics
in water resources, hydraulic machinery, oil transport networks,
flow in engines etc. Well-posedness theorem is an essential step for
applications in optimal control theory (Sritharan \cite{SS9}),
convergence of numerical algorithms and nonlinear filtering
(Sritharan \cite{SS8}, Fernando and Sritharan \cite{FS}).
Solvability theory of generalized solutions to Navier-Stokes
equations was pioneered by Leray \cite{Le}, Hopf \cite{Hp} and
Ladyzhenskaya \cite{La1}, \cite{La}. Steady state flow through
channels of various kinds has been studied by a number of authors
including Amick \cite{Am1}, \cite{Am2}, Amick and Franenkel
\cite{AF}, Ladyzhenskaya and Solonnikov \cite{LS}, \cite{LS1},
\cite{LS2}. In \cite{Am1}, Amick discussed the steady flow of
viscous incompressible fluid in channels and pipes in two and three
dimensions which are cylindrical outside some compact set. The paper
by Heywood \cite{He1} highlighted the question of uniqueness of the
solution of the Navier-Stokes equations for certain unbounded
domains modeling channels, tubes, or conduits of some kind and the
importance of prescribing flux or the overall pressure difference.
In \cite{AF}, Amick and Fraenkel studied steady state solutions of
the Navier-Stokes equations in various types of two dimensional
channel domains. In \cite{He2}, Heywood constructed classical
solutions of the Navier-Stokes equations for both stationary and
non-stationary boundary value problems in arbitrary
three-dimensional domains with smooth boundaries. The time dependent
flow through the three-dimensional channels with outlets diverging
at infinity has been studied by Ladyzhenskaya and Solonnikov
\cite{LS} and, Solonnikov \cite{So1}. The paper by Solonnikov
\cite{So1} presents solvability of boundary value problems for the
Stokes and Navier-Stokes equations in noncompact domains with
several oulets to infinity.  Babin \cite{Ba1} considered the
Navier-Stokes system in an unbounded planar channel-like domain and
proved that when the external force decays at infinity, the
semigroup generated by the system has a global attractor and its
Hausdroff dimension is finite using weighted Sobolev estimates. The
paper by Sritharan \cite{SS4} addressed the following two important
cases which were not considered in the earlier works:
\begin{itemize}
\item [(i)] time-dependent flow through two and three dimensional
channels with finite cross section;
\item [(ii)] time-dependent flow through two-dimensional channels
with outlets diverge at infinity
\end{itemize}
and provided a unique solvability theorem for the two-dimensional
case of the problem type (i).
\begin{figure}[b]
\begin{center}
\includegraphics[width=0.5\textwidth]{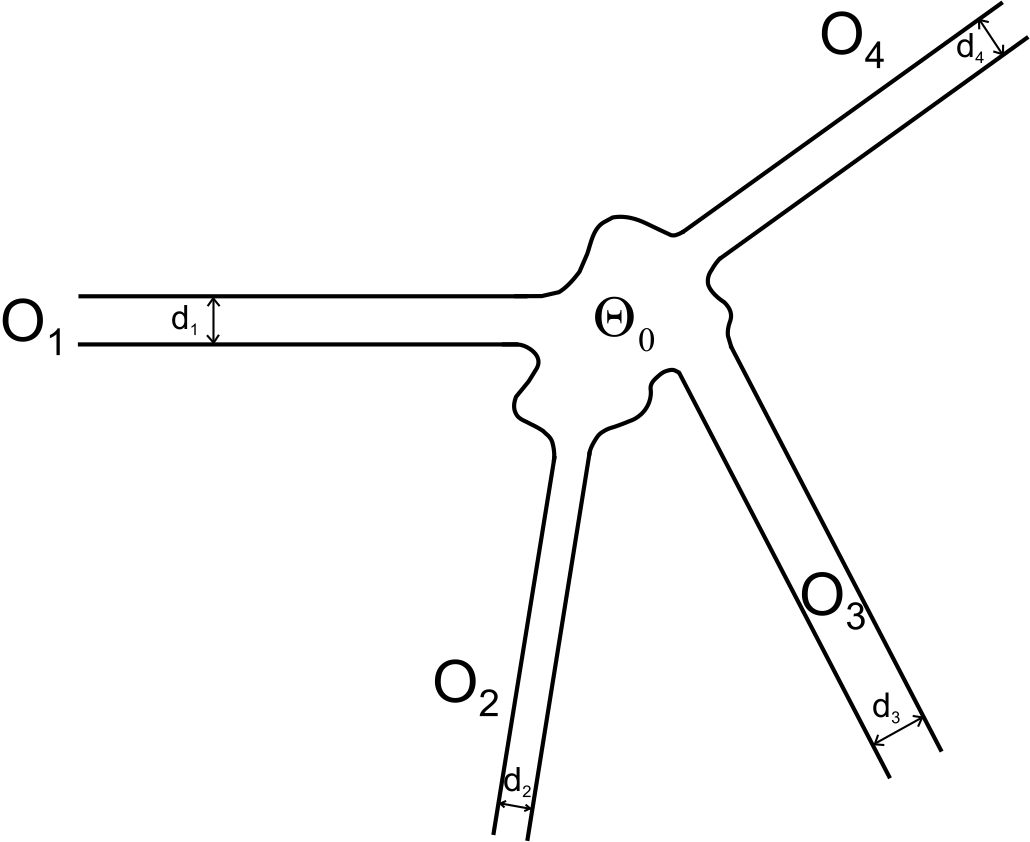}
\end{center}
\caption{multi-channel domain } \label{acd1}
\end{figure}
Solvability of stochastic Navier-Stokes equations in unbounded
channel-like domains with non-zero flux condition have remained as
an open problem in both two and three-dimensions. To the best of the
authors knowledge, this work appears to be the first systematic
treatment of stochastic two-dimensional Navier-Stokes equations in
such domains. In this paper we consider a stochastic version of the
problem of type (i) in multi-channel domains in two dimensions and
prove a unique solvability theorem with a possible future extension
to three dimensions (up to a stopping time determined by the size of
the flux and the Reynolds number). The problem of type (ii) may
possibly be resolved by suitably choosing a conformal mapping (see
Amick and Franenkel \cite{AF} for similar ideas in the case of
steady flows) to straighten the diverging outlets.

Let us consider the unbounded multi-channel domains, with several
outlets as shown in the figure (see Figure \ref{acd1}). Let the
outlets of the multi-channel domain be named as $\mathbf{O}_1,
\mathbf{O}_2,\cdots,\mathbf{O}_N$ and outside a compact region let
the outlets be of constant widths $d_1,d_2,\cdots,d_N$. Our first
step is to construct a basic vector field through each of these
outlets with the stochastic flux $\mathscr{F}_i(t,\omega)$ such that
$\sum_{i=1}^N\mathscr{F}_i(t,\omega)=0$, $\mathbb{P}$ - a. s. The
methodology of proof can be understood by considering a channel with
two outlets having a unit width connected in a smooth way
$\Theta=\mathbf{O}_1\cup\mathbf{O}_2\cup\Theta_0$ (see Figure
\ref{acd2}).
\begin{figure}[h]
\begin{center}
\includegraphics[width=0.5\textwidth]{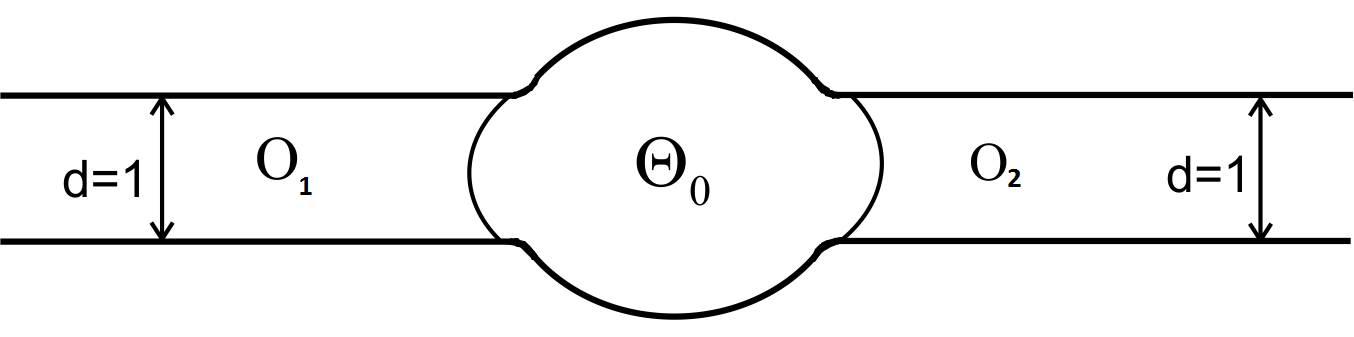}
\end{center}
\caption{channel with two outlets having unit width} \label{acd2}
\end{figure}

Let us now discuss the problems of type (i) and examine the
difficulties that arise in proving solvability. The time-dependent
Navier-Stokes problem is usually treated using the method of Hopf
\cite{Hp} in the deterministic setting which relies on
$\mathrm{L}^2$-energy estimates. The traditional methods of
solvability fail in the absence of an energy inequality. For
channels of finite cross section, as pointed out in Sritharan
\cite{SS4}, in order for the net flux to be nontrivial, the velocity
field should not decay to zero at infinity (upstream and downstream)
and hence such velocity fields would then have infinite energy.

Below we give a heuristic argument regarding the infinite energy of
such velocity fields. We also point out that in the absence of a
rigorous decay theory, only a heuristic argument could be made in
this regard.

Let the velocity field be stochastic and modeled on a complete
probability space ($\Omega, \mathcal{F},\mathcal{F}_t, \mathbb{P}$).
If the stochastic net flux of the velocity field $\u$ is
$\mathscr{F}(t,\omega)$, $\omega\in\Omega$, then across any cross
section $\Gamma$, we have
$$\int_{\Gamma}\u\cdot\n\d S=\mathscr{F}(t,\omega),$$
where $\n$ is the normal to the curve $\Gamma$ and $\d S$ is the
length element.

In this case, the velocity $|\u(\x,t,\omega)|\nrightarrow 0$ as
$|x|\to\infty$, (where $\x=(x,y)$ with $y$ is of constant width)
$\mathbb{P}$-a. s. To see this let us take the $2$-D case with the
outlet $\mathbf{O}_2=\left\{(x,y)\in(0,\infty)\times(0,1)\right\}$.
The flux across any cross section is same throughout the channel,
due to divergence free condition. That is, if $|\u(\x,t,\omega)|\to
0$ as $x\to\infty$, then the flux at the far field is zero. Hence
the flux across any cross section is zero throughout the channel
giving the net flux is zero. Thus for the flux to be non-zero, we
need the condition that $|\u(\x,t,\omega)|\nrightarrow 0$ as
$x\to\infty$.

As \textit{an
 example}, we
consider the decay of $\u(x,y,t,\omega)$ in the following form
 $$\u(x,y,t,\omega)\sim
 \frac{C_1(y,t,\omega)}{|x|^{\alpha}}+C_2(y,t,\omega),\alpha>1$$ for
 sufficiently large $M$ with $|x|>M$, where $C_1(y,t,\omega)$ and
 $C_2(y,t,\omega)>0$ be $\mathrm{L}^2(\Omega;C(0,T;\mathrm{L}^2(0,1)))$ functions such that
$$\int_0^1C_1(y,t,\omega)\d y=0\textrm{ and
 }\int_0^1C_2(y,t,\omega)\d y=\mathscr{F}(t,\omega).$$ From
 above, we have $\u(x,y,t,\omega)$ is flux carrying in the far field as $$\int_0^{1}\u(x,y,t,\omega)\d y\sim
 \int_0^1\frac{C_1(y,t,\omega)}{|x|^{\alpha}}\d
 y+\int_0^1C_2(y,t,\omega)\d y=\mathscr{F}(t,\omega).$$  Then at any given time $t\in(0,T]$, the $\mathrm{L}^2$-energy is given by
\begin{align}\label{int1}
&\int_0^1\int_0^{\infty}|\u(x,y,t,\omega)|^2\d x\d y \geq
\int_0^1\int_M^{\infty}|\u(x,y,t,\omega)|^2\d x\d y\nonumber\\&\sim
\int_0^1\int_M^{\infty}\left|\frac{C_1(y,t,\omega)}{|x|^{\alpha}}+C_2(y,t,\omega)\right|^2\d
x\d y\nonumber\\
&=
 \int_0^1\int_M^{\infty}\frac{|C_1(y,t,\omega)|^2}{|x|^{2\alpha
}}\d x\d y+\int_0^1\int_M^{\infty}|C_2(y,t,\omega)|^2\d x\d
y\nonumber\\&\quad+2\int_0^1\int_M^{\infty}\frac{C_1(y,t,\omega)C_2(y,t,\omega)}{|x|^{\alpha}}\d
x\d y.
\end{align}
The first integral in (\ref{int1}) always has a finite positive
value  and is equal to
$$\frac{1}{(2\alpha-1)M^{2\alpha-1}}\left(\int_0^1|C_1(y,t,\omega)|^2\d
y\right).$$ The second integral diverges to $+\infty$. The last
integral in (\ref{int1}) is also bounded, since
\begin{align*}
&\left|\int_0^1\int_M^{\infty}\frac{C_1(y,t,\omega)C_2(y,t,\omega)}{|x|^{\alpha}}\d
x\d y\right|\nonumber\\&\leq
\int_0^1\int_M^{\infty}\frac{|C_1(y,t,\omega)||C_2(y,t,\omega)|}{|x|^{\alpha}}\d
x\d y\nonumber\\ &\leq
\frac{1}{(\alpha-1)M^{\alpha-1}}\left(\int_0^1|C_1(y,t,\omega)|^2\d
y\right)^{1/2}\left(\int_0^1|C_2(y,t,\omega)|^2\d y\right)^{1/2}.
\end{align*}
Thus for any given time $t>0$, we have
$$\int_0^1\int_0^{\infty}|\u(x,y,t,\omega)|^2\d x\d y=\infty,\mathbb{P}-\textrm{ a. s.}$$

There are extensive literature on deterministic flow through channel
type domains. Interested readers may look into Amick \cite{Am1,
Am2}, Amick and Franenkel \cite{AF}, Babin \cite{Ba1, Ba2}, Borchers
and Pileckas \cite{BoP}, Heywood \cite{He1,He2, He4},
Kapitanski\`{\i} and Piletskas \cite{KaP}, Ladyzhenskaya and
Solonnikov \cite{LS, LS1, LS2}, Pileckas \cite{PlK}, Piletskas
\cite{PKI}, Solonnikov \cite{So1}, Solonnikov and Piletskas
\cite{SoPl}, Sritharan \cite{SS7,SS4} to name a few.

For a sample of literature on stochastic Navier-Stokes equations, we
refer the readers to Bensoussan \cite{Be1}, Bensoussan and Temam
\cite{BT}, Capinski and Cutland \cite{CC}, Da Parto and Zabczyk
\cite{DaZ},  Flandoli and Gatarek \cite{FG},  Menaldi and Sritharan
\cite{MS}, Pardoux \cite{Pa1}, \cite{Pa2},  Sritharan and Sundar
\cite{SS2},   Vishik and Fursikov \cite{VF}, Sritharan \cite{SS9},
Fernando and Sritharan \cite{FS}, Sakthivel and Sritharan
\cite{SaS}.

The plan of the paper is as follows. In section 2, the main result
of this paper and the functional setting have been given. A
divergence free vector field of infinite energy carrying a
nontrivial net flux  through the channel is constructed in section 3
using the solution of the heat equation. In section 4, we
characterize the properties of the linear and bilinear operators
that are associated with the Navier-Stokes problem. A perturbed
vector field is constructed in section 5 using a suitable
transformation involving the constructed basic vector field.
A-priori estimates for the solutions of the perturbed vector field
are obtained in section 6. In section 7, we prove the local
monotonicity condition for the sum of the Stokes and the inertia
operators as well as the existence and uniqueness of strong
solutions to the perturbed vector field by exploiting this local
monotonicity condition. In section 8, we mathematically characterize
the perturbation pressure field using a generalization of the de
Rham's Theorem to processes. section 9, completes the proof of the
main result.

\section{Basic Definitions and the Main Theorem}
In this section, following Sritharan \cite{SS4}, we define the class
of channel domains that will be analyzed.
\begin{definition}(Admissible channel domain)
A simply connected open set $\Theta\subset \mathbb{R}^2$ with
$C^{\infty}$ boundary $\partial \Theta$ consisting of two
disconnected components $\partial\Theta_1$ and $\partial\Theta_2$ is
called an admissible channel domain (see Figure \ref{acd}), if it is
the union of three disjoint sets
$\Theta_0\cup\mathbf{O}_1\cup\mathbf{O}_2$ defined in the following
way. Let $\mathbf{O}_1$ and $\mathbf{O}_2$ be two semi-infinite
strips of width $d_1$ and $d_2$ respectively. These two straight
channels are smoothly (not necessarily coaxially) joined by a
bounded domain $\Theta_0$ such that
$\partial\Theta_1\cup\partial\Theta_2=\partial\Theta\in C^{\infty}.$
\end{definition}
\begin{figure}[h]
\begin{center}
\includegraphics[width=0.5\textwidth]{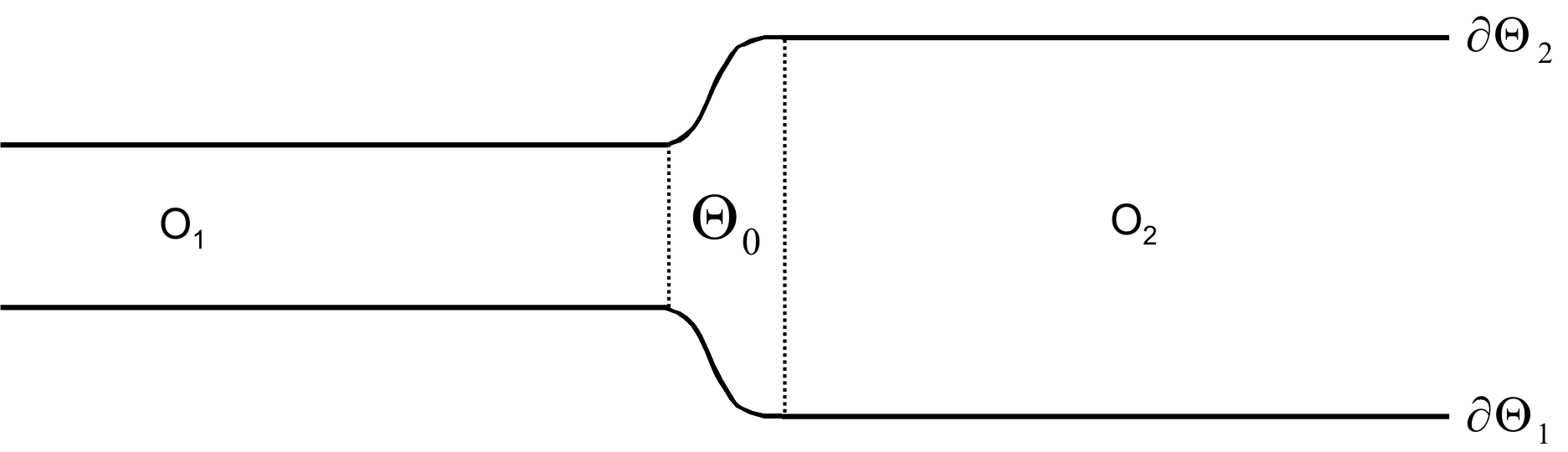}
\end{center}
\caption{admissible channel domain} \label{acd}
\end{figure}
Now let us consider the problem of accelerating a viscous
incompressible fluid from rest to a given stochastic flux rate
through an admissible channel domain. Let
$\left(\Omega,\mathcal{F},\mathcal{F}_t,\mathbb{P}\right)$ be a
complete probability space. The mathematical problem is to find the
velocity field $\u$ and pressure field $p$ such that
$$(\mathbf{u},p):\Theta\times [0,T]\times\Omega\rightarrow \mathbb{R}^2\times
\mathbb{R},$$ the momentum equation
\begin{align}\label{1}
\mathbf{u}_t+\mathbf{u}\cdot\nabla\mathbf{u}=-\nabla
p+\nu\Delta\mathbf{u}+\dot{\mathscr{G}}(\x,t,\omega)\textrm{ in
}\Theta\times (0,T)\times \Omega,
\end{align}
the incompressibility condition
\begin{align}\label{2}
\nabla\cdot \mathbf{u} =0\textrm{ in } \Theta\times
(0,T)\times\Omega,
\end{align}
the non-slip boundary condition on the channel walls
\begin{align}\label{3}
\mathbf{u}(\mathbf{x},t,\omega)=0\textrm{ on } \partial \Theta\times
[0,T]\times\Omega,
\end{align}
the initial condition
\begin{align}\label{4}
\mathbf{u}(\mathbf{x},0,\omega)=0,\;\;(\mathbf{x},\omega)\in\Theta\times\Omega,
\end{align}
and the flux condition
\begin{align}\label{5}
\int_{\Gamma}\mathbf{u}\cdot\mathbf{n}\d
S=\mathscr{F}(t,\omega),\textrm{ for all }t\in [0,T]\textrm{ with
}\mathscr{F}(0,\omega)=0,\;\forall\;\omega\in\Omega,
\end{align}
are satisfied. The properties of the stochastic flux will be
discussed in the later sections. Here $\nu> 0$ is the coefficient of
kinematic viscosity and $\Gamma$ is any cross-sectional curve
cutting the channel.

In this formulation the stochasticity of fluid flow is due to an
external random forcing and the random flux. Also we will assume
that the external random forcing $\dot{\mathscr{G}}(\x,t,\omega)$
and the random flux $\mathscr{F}(t,\omega)$ are mutually independent
processes. Further details about the noise have been provided in the
subsequent sections.

Now let us state the main result of this paper.
\begin{theorem}\label{thm111}
Suppose that the flux rate satisfies the moment bound:
\begin{align}\label{6}
\mathbb{E}\left(\int_0^T\left|\frac{\partial}{\partial
t}\mathscr{F}(t,\omega)\right|^2\d t\right)\leq C_1(T),
\end{align}
for some prescribed $T$. Then for each such
$\mathscr{F}(\cdot,\omega)$ there exits a unique strong solution
$\mathbf{u}(\x,t,\omega)$ with the following estimates:
\begin{align}\label{7}
&\mathbb{E}\left(\sup_{0\leq t\leq
T}\int_{\Theta}\left|(\u-\w)(\x,t,\cdot)\right|^2\d \x+2\nu
\int_0^T\int_{\Theta}|\nabla(\u-\w)(\x,t,\cdot)|^2\d \x\d
t\right)\nonumber\\&\qquad\qquad\qquad\leq C\left(T,\nu,\int_0^T
\mathrm{Tr}(g^*g)(t)\d t\right),
\end{align}
for some divergence free vector field $\w(\x,t,\omega)$ that
vanishes on the boundary $\partial \Theta$ and carries the
prescribed flux
\begin{align}
\int_{\Gamma}\u\cdot\n\d S=\int_{\Gamma}\w\cdot\n\d
S=\mathscr{F}(t,\omega),\;\forall\;\omega\in\Omega.
\end{align}
\end{theorem}
We prove the above theorem in the subsequent sections. The following
functional frame work is used in this paper.
\begin{align*}
C_0^{\infty}(\Theta) &= \textrm{ the space of all infinitely
differentiable
vector fields with }\\&\textrm{\quad\quad compact support in } \Theta,\\
W_0(\Theta)&=\textrm{ the completion of }
C_0^{\infty}(\Theta)\textrm{ vector fields in the norm } \|\nabla
\varphi\|_{\mathrm{L}^2(\Theta)},\\H^1(\Theta) &=
\left\{\varphi:\Theta\rightarrow\mathbb{R}^2;\;\varphi\in\mathrm{L}^2(\Theta),\;\nabla\varphi\in\mathrm{L}^2(\Theta)\right\}\!\!,\nonumber\\
 H_0^1(\Theta) &=
\left\{\varphi:\Theta\rightarrow\mathbb{R}^2;\;\varphi\in\mathrm{L}^2(\Theta),\;\nabla\varphi\in\mathrm{L}^2(\Theta)\;\mathrm{
and }\; \varphi\big|_{\partial\Theta}=0\right\}\!\!,\nonumber\\
j(\Theta)& =
\left\{\varphi:\Theta\rightarrow\mathbb{R}^2;\;\varphi\in
C_0^{\infty}(\Theta)\;\mathrm{ and }\;\nabla\cdot\varphi
=0\right\}\!.
\end{align*}
Also let us define
\begin{align*}
\mathbf{H}&= \textrm{Completion of } j(\Theta) \textrm{ in the }
\mathrm{L}^2(\Theta) \textrm{ norm},\\
\mathbf{V} &= \textrm{Completion of }
j(\Theta) \textrm{ in the } H^1(\Theta) \textrm{ norm},\\
\mathbf{H}^*&=\left\{\varphi:\Theta\rightarrow
\mathbb{R}^2;\;\varphi\in\mathrm{L}^2(\Theta);\;\nabla\cdot\varphi=0\textrm{
and }\varphi\cdot\n\big|_{\partial\Theta}=0\right\}\!\!,\\
\mathbf{V}^*&=\left\{\varphi:\Theta\rightarrow\mathbb{R}^2;\;\varphi\in
H_0^1(\Theta);\textrm{ and } \nabla\cdot\varphi=0\right\}\!\!,\\
\mathbf{V}_0&=\textrm{completion of } j(\Theta) \textrm{
in the norm of }\|\nabla\varphi\|_{\mathrm{L}^2(\Theta)},\\
\mathbf{V}_0^*&=\left\{\varphi\in
W_0(\Theta);\nabla\cdot\varphi=0\right\}\!.
\end{align*}

Let us denote the norm in $\mathbf{H}$ by $|\cdot|$ and the norm in
$\mathbf{V}$ by $\|\cdot\|$. If we identify $\mathbf{H}$ with its
dual $\mathbf{H}'=\mathscr{L}(\mathbf{H};\mathbb{R})$ using the
Riesz representation theorem, we obtain the continuous and dense
embedding
$$\mathbf{V}\subset \mathbf{H}\equiv\mathbf{H}'\subset\mathbf{V}'.$$
Also let us denote the duality pairing between $\mathbf{V}$ and
$\mathbf{V}'$ by $(\cdot,\cdot)$.  Note, however, that (unlike in
bounded domains) the embedding $\mathbf{V}\subset\mathbf{H}$ is not
compact since $\Theta$ is unbounded. Poincar\'{e} lemma holds for
admissible channel domains (since they have finite cross section):
$$\|\phi\|_{\mathrm{L}^2(\Theta)}\leq
C\|\nabla\phi\|_{\mathrm{L}^2(\Theta)},\quad \forall\phi\in
H^1_0(\Theta)$$ and hence, in $\mathbf{V}$ the norm of $H^1(\Theta)$
is equivalent to that obtained by the Dirichlet integral
$\|\nabla\phi\|_{\mathrm{L}^2(\Theta)}.$

\section{Construction of the Basic Vector Field}\setcounter{equation}{0}
Following the ideas from Sritharan \cite{SS4}, we will construct a
divergence free basic vector field $\w(\x,t,\omega)$ in $\Theta$
which vanishes on $\partial\Theta$ and carries the prescribed random
flux $\mathscr{F}(t,\omega)$ through the channel. Note that this
``constructed" vector field need not satisfy the Navier-Stokes
equations in $\Theta$ although it does in $\mathbf{O}_1$ and
$\mathbf{O}_2$ due to the nature of the construction used. The
method can be described as follows. Using the solution of the
one-dimensional heat equation with random flux
$\mathscr{F}(t,\omega)$, a vector field is constructed in each of
the straight channel outlets $\mathbf{O}_1$ and $\mathbf{O}_2$. We
then smoothly join these two vector fields by constructing a smooth
extension in $\Theta_0$. In $\mathbf{O}_1$ and $\mathbf{O}_2$ the
vector field will have only one component, namely in the direction
of the axes of the outlets. In $\Theta_0$, however, in general both
components of $\w(\x,t,\omega)$ will be nonzero. Let us now consider
one of the outlets, say $\mathbf{O}_2$ and assume for simplicity
that the width of the channel is unity (see Figure \ref{fig2}).

\begin{figure}[h]
\begin{center}
\includegraphics[width=.5\textwidth]{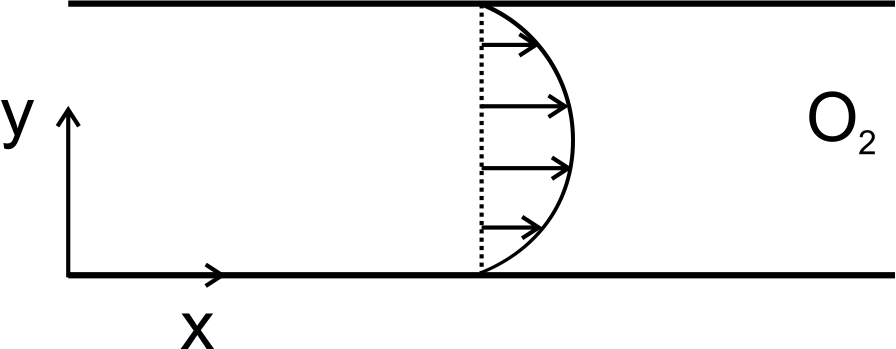}
\end{center}
\caption{The outlet $\mathbf{O}_2$} \label{fig2}
\end{figure}

Let us define the outlet
$\mathbf{O}_2=\left\{(x,y)\in(0,\infty)\times(0,1)\right\}.$ Let us
seek a divergence-free vector field in the form
$\mathbf{w}(x,y,t,\omega)=(w_1(y,t,\omega),0)$ and a scalar field
$P(x,t,\omega)$ in $\mathbf{O}_2$ such that
\begin{align}\label{eqn1}
\frac{\partial}{\partial t}w_1(y,t,\omega)-\nu
\frac{\partial^2}{\partial
y^2}w_1(y,t,\omega)&=-\frac{\partial}{\partial
x}P(x,t,\omega)=f(t,\omega)\nonumber\\&
\quad\textrm{ in }\mathbf{O}_2\times (0,T)\times\Omega,\\
w_1(0,t,\omega)=w_1(1,t,\omega)&=0 \textrm{ for } t\in [0,T],\omega\in\Omega,\\
w_1(y,0,\omega)&=0\textrm{ for }y\in(0,1),\omega\in\Omega\textrm{
and }\label{eqn01}
\end{align}
\begin{equation}\label{eqn2}
\int_0^1w_1(y,t,\omega)\d y=\mathscr{F}(t,\omega)\textrm{ for }
t\in[0,T] \textrm{ with
}\mathscr{F}(0,\omega)=0\;\forall\;\omega\in\Omega.
\end{equation}
Here the function $f(t,\omega)$ needs to be determined from the
prescribed flux $\mathscr{F}(t,\omega)$. To resolve this problem we
first write down the solution of the system
(\ref{eqn1})-(\ref{eqn01}) in terms of $f(t,\omega)$ and then use
the condition (\ref{eqn2}) to evaluate $f(t,\omega)$ in terms of
$\mathscr{F}(t,\omega)$. The solution of (\ref{eqn1})-(\ref{eqn01})
can be obtained by the method of separation of variables (see Cannon
\cite{Ca}). The existence and uniqueness of solution of the boundary
value problem for the heat equation with stochastic boundary
conditions  has been proved in Cahlon \cite{CB}.

\begin{theorem}
Let $f(t,\omega)$ satisfies the following  uniform H\"{o}lder
condition on $[0,T]$, for all $t_1,t_2\in[0,T]$,
\begin{align}\label{en10}
|f(t_1,\omega)-f(t_2,\omega)|\leq
L(\omega)|t_1-t_2|^{\gamma},0<\gamma< \frac{1}{2},
\end{align}
with $L(\omega)>0$ and $L(\cdot)\in\mathrm{L}^2(\Omega)$. Then,
there exists a unique solution $w_1(\cdot,\cdot,\cdot)$ of the
problem (\ref{eqn1})-(\ref{eqn01}) such that
$$w_1(y,t,\omega)\in\mathrm{L}^2\left(\Omega;C([0,1]\times [0,T])\right)$$
and satisfying the following conditions.
\begin{enumerate}
\item [(i)] $w_1(y,t,\omega)$ is a continuous function of $y$ and $t$ in
$\mathrm{L}^2(\Omega)$. i.e.,
\begin{align}
&\mathbb{E}\left[\left(w_1(y,t+h,\omega)-w_1(y,t,\omega)\right)^2\right]\to 0\textrm{ as }h\to 0,\nonumber\\
&\mathbb{E}\left[\left(w_1(y+h,t,\omega)-w_1(y,t,\omega)\right)^2\right]\to
0\textrm{ as }h\to 0.
\end{align}
\item [(ii)] There exists a stochastic function
$w_{1_{t}}(y,t,\omega)\in\mathrm{L}^2(\Omega;C([0,1]\times [0,T]))$
such that
\begin{align}\label{aq1}
\mathbb{E}\left[\left(\frac{w_1(y,t+h,\omega)-w_1(y,t,\omega)}{h}-w_{1_{t}}(y,t,\omega)\right)^2\right]\to
0\textrm{ as }h\to 0.
\end{align}
\item [(iii)] There exists a stochastic function
$w_{1_{yy}}(y,t,\omega)\in\mathrm{L}^2(\Omega;C([0,1]\times [0,T]))$
such that
\begin{align}
&\mathbb{E}\left[\left(\frac{w_1(y+h,t,\omega)+w_1(y-h,t,\omega)-2w_1(y,t,\omega)}{h^2}-w_{1_{yy}}(y,t,\omega)\right)^2\right]\nonumber\\&\to
0\textrm{ as } h\to 0.
\end{align}
\item [(iv)] The equation
$w_{1_{t}}(y,t,\omega)=\nu w_{1_{yy}}(y,t,\omega)+f(t,\omega)$ is
satisfied for a. e. $\omega\in\Omega$.
\end{enumerate}
\end{theorem}

\begin{proof}
By using the method of separation of variables, we can write down
the solution of (\ref{eqn1})-(\ref{eqn01}) as
\begin{equation}\label{eqn3}
w_1(y,t,\omega)=\frac{4}{\pi}
\sum_{n=0}^{\infty}\left\{\int_0^tf(s,\omega)\frac{1}{2n+1}e^{(-\nu(2n+1)^2\pi^2(t-s))}\d
s\right\} \sin(2n+1)\pi y.
\end{equation}

We now use  the Weierstrass' M-test to  show that the  infinite
series solution (\ref{eqn3}) to the Heat Equation
(\ref{eqn1})-(\ref{eqn01}) converges uniformly. Let us take
$$a_n(y,t,\omega)=
\int_0^tf(s,\omega)\frac{1}{2n+1}e^{(-\nu(2n+1)^2\pi^2(t-s))}\d s
\sin(2n+1)\pi y,\textrm{ for }t>0.$$ Then by using H\"{o}lder's
inequality, for $t\in(0,T]$ and $y\in  [0,1]$, we have
\begin{align}\label{1a1}
|a_n(y,t,\omega)|&\leq
\int_0^t|f(s,\omega)|\frac{1}{2n+1}e^{(-\nu(2n+1)^2\pi^2(t-s))}\d
s\nonumber\\&\leq
\frac{1}{(2n+1)}e^{-\nu(2n+1)^2\pi^2t}\left(\int_0^t|f(s,\omega)|^2\d
s\right)^{1/2}\left(\int_0^te^{2\nu(2n+1)^2\pi^2s}\d
s\right)^{1/2}\nonumber\\&\leq
\|f(\cdot,\omega)\|_{\mathrm{L}^2(0,T)}\frac{1}{(2n+1)}e^{-\nu(2n+1)^2\pi^2t}\left(\frac{e^{2\nu(2n+1)^2\pi^2t}-1}{2\nu(2n+1)^2\pi^2}\right)^{1/2}\nonumber\\
&\leq
\|f(\cdot,\omega)\|_{\mathrm{L}^2(0,T)}\frac{1}{(2n+1)^2\pi\sqrt{2\nu}},\quad\mathbb{P}-\textrm{
a. s. }
\end{align}
Hence, we get $$|a_n(y,t,\omega)|\leq M_n\textrm{ where
}M_n=\|f(\cdot,\omega)\|_{\mathrm{L}^2(0,T)}\frac{1}{(2n+1)^2\pi\sqrt{2\nu}},\mathbb{P}-\textrm{
a. s.}$$ Since
$\sum_{n=0}^{\infty}\frac{1}{(2n+1)^2}=\frac{\pi^2}{8}$ (from
Apostol \cite{TMA}, we have
$\sum_{n=1}^{\infty}\frac{1}{n^2}=\frac{\pi^2}{6}$, then
$\sum_{n=1}^{\infty}\frac{1}{(2n)^2}+\sum_{n=0}^{\infty}\frac{1}{(2n+1)^2}=\frac{\pi^2}{6}$
gives $\sum_{n=0}^{\infty}\frac{1}{(2n+1)^2}=\frac{\pi^2}{8}$), we
have
$\sum_{n=0}^{\infty}M_n=\frac{\pi}{8\sqrt{2\nu}}\|f(\cdot,\omega)\|_{\mathrm{L}^2(0,T)}$,
$\mathbb{P}$- a. s., for all $t\in (0,T]$. Hence by the Weierstrass'
M-test,
$w_1(y,t,\omega)=\frac{4}{\pi}\sum_{n=0}^{\infty}a_n(y,t,\omega)$
converges absolutely and uniformly for $y\in [0, 1]$ and $t>0$. For
$t=0$ also the convergence is uniform since $w_1(y,0,\omega)=0$.

One can re-write (\ref{eqn3}) as
\begin{equation}
w_1(y,t,\omega)=\int_0^tK(y,t-s)f(s,\omega)\d s=(K*f)(t),
\end{equation}
where
$$K(y,t)=\frac{4}{\pi}\sum_{n=0}^{\infty}\frac{1}{(2n+1)}e^{-\nu(2n+1)^2\pi^2t}\sin(2n+1)\pi
y.$$ From (\ref{en10}), we have
\begin{align*}
\mathbb{E}\left[|f(t_1,\omega)-f(t_2,\omega)|^2\right]&\leq
\mathbb{E}\left[L(\omega)|t_1-t_2|^{\gamma}\right]^2\nonumber\\&=\mathbb{E}|L(\omega)|^2|t_1-t_2|^{2\gamma}
\to 0 \textrm{ as } t_1\to t_2,
\end{align*}
for all $t_1,t_2\in[0,T]$, which gives
$f(\cdot,\cdot)\in\mathrm{L}^2(\Omega;C[0,T]).$  We can prove that
the solution (\ref{eqn3}) of the problem (\ref{eqn1})-(\ref{eqn01})
satisfies $w_1(y,t,\omega)\in\mathrm{L}^2\left(\Omega;C([0,1]\times
[0,T])\right)$ (see Cahlon \cite{CB} and Walsh \cite{Wa} for
existence and uniqueness).

Let us now prove part (iv) of the theorem. For $t>0$, we have
\begin{align}\label{en11}
\frac{\partial}{\partial
t}(w_1(y,t,\omega))&=\frac{\partial}{\partial
t}\left[\int_0^tK(y,t-s)f(s,\omega)\d s\right]\nonumber\\
&=\frac{\partial}{\partial
t}\left[\int_0^{\infty}H(t-s)K(y,t-s)f(s,\omega)\d
s\right]\nonumber\\&=\int_0^{\infty}\left[H(t-s)\frac{\partial}{\partial
t}K(y,t-s)+K(y,t-s)\delta(t-s)\right]f(s,\omega)\d
s\nonumber\\&=\int_0^t\frac{\partial}{\partial
t}K(y,t-s)f(s,\omega)\d
s+\frac{4}{\pi}\sum_{n=0}^{\infty}\frac{\sin(2n+1)\pi
y}{(2n+1)}f(t,\omega)\nonumber\\
&=\int_0^t\frac{\partial}{\partial t}K(y,t-s)f(s,\omega)\d
s+f(t,\omega),
\end{align}
where $H(\cdot)$ is the Heaviside function. The last term is
obtained by using the property of $\delta$-function and using the
formula $\sum_{n=0}^{\infty}\frac{\sin(2n+1)\pi
y}{(2n+1)}=\frac{\pi}{4}$, for $y\in(0,1)$. By a direct calculation,
as $K(\cdot,\cdot)$ is uniformly convergent in time $t>0$ for a
fixed $y\in(0,1)$, it is easy to see that
$$\frac{\partial}{\partial
t}K(y,t-s)=-4\nu\pi\sum_{n=0}^{\infty}(2n+1)e^{-\nu(2n+1)^2\pi^2(t-s)}\sin(2n+1)\pi
y.$$ From (\ref{en11}), using the uniform convergence of the above
series in time $t>0$ for a fixed $y\in(0,1)$, we have
\begin{align}\label{en12}
\frac{\partial}{\partial
t}(w_1(y,t,\omega))&=-4\nu\pi\sum_{n=0}^{\infty}\left\{\int_0^t\left[(2n+1)e^{-\nu(2n+1)^2\pi^2(t-s)}f(s,\omega)\right]\d
s\times\right. \nonumber\\&\quad\sin(2n+1)\pi y\Big\}+f(t,\omega).
\end{align}
Next, let us calculate $\frac{\partial^2}{\partial
y^2}w_1(y,t,\omega)$. We have
\begin{align}\label{en13}
\frac{\partial^2}{\partial
y^2}w_1(y,t,\omega)=\int_0^t\frac{\partial^2}{\partial
y^2}[K(y,t-s)]f(s,\omega)\d s.
\end{align}
Since $K(\cdot,\cdot)$ is uniformly convergent in $y$ and
$\frac{\partial K}{\partial t}(\cdot,\cdot)$ is also uniformly
convergent in $y$ for a fixed $t>0$, by a simple calculation, we
have
$$\frac{\partial^2}{\partial
y^2}[K(y,t-s)]=-4\pi\sum_{n=0}^{\infty}(2n+1)e^{-\nu(2n+1)^2\pi^2(t-s)}\sin(2n+1)\pi
y.$$ The above series is uniformly convergent in time $t>0$ for a
fixed $y\in(0,1)$ and hence from (\ref{en13}), we get
\begin{align}\label{en14}
\frac{\partial^2w_1}{\partial
y^2}=-4\pi\sum_{n=0}^{\infty}\int_0^t(2n+1)e^{-\nu(2n+1)^2\pi^2(t-s)}f(s,\omega)\sin(2n+1)\pi
y\d s.
\end{align}
The above integral is well defined and from equations (\ref{en12})
and (\ref{en14}), we have part (iv) of the theorem.

Now we prove part (i) of the theorem. Since $K$ is continuous in
time $t$, for any given $\e>0$, there exists a $\eta>0$ such that
$$\left|K(y,t_1)-K(y,t_2)\right|<\e\textrm{ whenever
}|t_1-t_2|<\eta,$$ $t_1,t_2\in (0,T]$. For
$f(\cdot,\cdot)\in\mathrm{L}^2(\Omega;C[0,T])$, by choosing
$|h|<\eta$, let us consider
$\mathbb{E}\left[w_1(y,t+h,\omega)-w_1(y,t,\omega)\right]^2$ and by
using Young's inequality and continuity of $K(y,t)$ in $t$, we have
\begin{align*}
&\mathbb{E}\left[w_1(y,t+h,\omega)-w_1(y,t,\omega)\right]^2\nonumber\\
&=\mathbb{E}\left(\int_0^{t+h}K(y,t+h-s)f(s,\omega)\d
s-\int_0^tK(y,t-s)f(s,\omega)\d s\right)^2\nonumber\\&
\leq2\mathbb{E}\left(\int_0^t\left|K(y,t+h-s)-K(y,t-s)\right|^2|f(s,\omega)|^2\d
s\right)\nonumber\\&\quad+2\mathbb{E}\left(\int_t^{t+h}|K(y,t+h-s)|^2|f(s,\omega)|^2\d
s\right)\nonumber\\&\leq2\mathbb{E}\left[\sup_{0\leq s\leq
t}|f(s,\omega)|^2\right]\int_0^t|K(y,t+h-s)-K(y,t-s)|^2\d
s\nonumber\\&\quad+2\mathbb{E}\left[\sup_{0\leq s\leq
t}|f(s,\omega)|^2\right]\int_t^{t+h}|K(y,t+h-s)|^2\d s
\nonumber\\&\leq
2\|f\|^2_{\mathrm{L}^2(\Omega;C[0,T])}\left(\e^2T+\int_t^{t+h}|K(y,t+h-s)|^2\d
s\right)\to 0\textrm{ as }h\to 0,
\end{align*}
since $\e$ is arbitrary and as $h\to 0$, $t\to s$ and $K(y,t+h-s)\to
K(y,0)=1$ for all $y\in(0,1)$. Similarly, $
\mathbb{E}(w_1(y+h,t,\omega)-w_1(y,t,\omega))^2\to 0\textrm{ as
}h\to 0. $

Let us now prove part (ii) of the theorem. Since for a fixed
$y\in(0,1)$, $K(y,t)$ is uniformly convergent in time $t>0$ and its
derivative $\frac{\partial K}{\partial t}$ exists and also is
uniformly convergent for  $t>0$, we have for a given $\e>0$, there
exists an $\eta>0$ such that
$$\left|\frac{K(y,t+h)-K(y,t)}{h}-\frac{\partial}{\partial t}K(y,t)\right|<\e\textrm{ when ever }|h|<\eta.$$ To prove
(\ref{aq1}), let us use the differentiability of $K(y,t)$ in time
$t$. For $|h|<\eta$, we have
\begin{align*}
&\mathbb{E}\left[\left(\frac{w_1(y,t+h,\omega)-w_1(y,t,\omega)}{h}-w_{1_{t}}(y,t,\omega)\right)^2\right]
\nonumber\\&=\mathbb{E}\left[\int_0^t\left(\frac{K(y,t+h-s)-K(y,t-s)}{h}\right)f(s,\omega)\d
s\right.\nonumber\\&\left.\quad+\frac{1}{h}\int_t^{t+h}K(y,t+h-s)f(s,\omega)\d
s-\int_0^t\frac{\partial}{\partial t}K(y,t-s)f(s,\omega)\d
s-f(t,\omega)\right]^2\nonumber\\&\leq 2\mathbb{E}\left[\sup_{0\leq
s\leq
t}|f(s,\omega)|^2\right]\int_0^t\left|\frac{K(y,t+h-s)-K(y,t-s)}{h}-\frac{\partial}{\partial
t}K(y,t-s)\right|^2\d
s\nonumber\\&\quad+2\mathbb{E}\left(\frac{1}{h}\int_t^{t+h}\left[K(y,t+h-s)f(s,\omega)-f(t,\omega)\right]\d
s\right)^2\nonumber\\&\leq 2
\|f\|^2_{\mathrm{L}^2(\Omega;C[0,T])}\e^2T+2\mathbb{E}\left(\frac{1}{h}\int_t^{t+h}\left[K(y,t+h-s)f(s,\omega)-f(t,\omega)\right]\d
s\right)^2.
\end{align*}
Note that in the time interval $[t, t+h]$ as $h \to 0$, $s \to t$
and hence the function $K(y, t+h-s) f(s, \omega) \to  K(y, 0) f(t,
\omega) = f(t, \omega)$. Thus by the Lebesgue's differentiation
theorem (Theorem 6, Appendix E.4 of Evans \cite{LCE}), the last term
of the right hand side of the above inequality goes to $0$ as $h\to
0$. Finally, since $\e> 0$ is arbitrary, we have the desired result
(\ref{aq1}). Similarly one can prove that there exists a stochastic
function $w_{1_{yy}}(y,t,\omega)\in\mathrm{L}^2(\Omega;C([0,1]\times
[0,T]))$ such that
\begin{align*}
&\mathbb{E}\left[\left(\frac{w_1(y+h,t,\omega)+w_1(y-h,t,\omega)-2w_1(y,t,\omega)}{h^2}-w_{1_{yy}}(y,t,\omega)\right)^2\right]\nonumber\\&\to
0\textrm{ as } h\to 0.
\end{align*}
\end{proof}

\begin{corollary}
For $f(\cdot,\cdot)\in\mathrm{L}^2(\Omega;C[0,T])$, there exists a
function $\frac{\partial}{\partial
t}\mathscr{F}(\cdot,\cdot)\in\mathrm{L}^2(\Omega;C[0,T])$ such that
\begin{align}\label{en21a} \frac{\partial}{\partial
t}\mathscr{F}(t,\omega)=f(t,\omega)+\int_0^t\frac{\partial}{\partial
t}\mathfrak{h}(t-\tau)f(\tau,\omega)\d \tau,
\end{align}
where $
\mathfrak{h}(t)=\frac{8}{\pi^2}\sum_{n=0}^\infty\frac{1}{(2n+1)^2}
\exp(-\nu(2n+1)^2\pi^2t). $
\end{corollary}

\begin{proof}
From \eqref{eqn2}, we have
\begin{align}\label{eqn4}
\mathscr{F}(t,\omega)&=\frac{8}{\pi^2}\sum_{n=0}^{\infty}\left\{\int_0^tf(\tau,\omega)\frac{1}{(2n+1)^2}e^{(-\nu(2n+1)^2\pi^2(t-\tau))}\d
\tau\right\}\nonumber\\&=\int_0^t\mathfrak{h}(t-\tau)f(\tau,\omega)\d
\tau,
\end{align}
where
\begin{equation}\label{eqn5}
\mathfrak{h}(t)=\frac{8}{\pi^2}\sum_{n=0}^\infty\frac{1}{(2n+1)^2}
\exp(-\nu(2n+1)^2\pi^2t).
\end{equation}
Note that since $\mathfrak{h}(t)$ is a decreasing function of $t$,
we have for all $t>0,\;\mathfrak{h}(t)\leq
\mathfrak{h}(0)=\frac{8}{\pi^2}\sum_{n=0}^\infty\frac{1}{(2n+1)^2}=1$.
Also the series $\mathfrak{h}(t)$ is convergent. Then there exists a
function $\frac{\partial}{\partial
t}\mathscr{F}(t,\omega)\in\mathrm{L}^2(\Omega;C[0,T])$ such that
\begin{align}\label{en21}
\frac{\partial}{\partial
t}\mathscr{F}(t,\omega)=f(t,\omega)+\int_0^t\frac{\partial}{\partial
t}\mathfrak{h}(t-\tau)f(\tau,\omega)\d \tau,
\end{align}
where $\frac{\partial}{\partial
t}\mathfrak{h}(t)=-8\nu\sum_{n=0}^{\infty}e^{-\nu(2n+1)^2\pi^2t}$.
Since $\mathfrak{h}$ is differentiable in time $t>0$, for any given
$\e>0$, there exists an $\eta>0$ such that
$$\left|\frac{\mathfrak{h}(t+h)-\mathfrak{h}(t)}{h}-\frac{\partial}{\partial t}\mathfrak{h}(t)\right|\leq
\e\textrm{ when ever }|h|<\eta.$$

For proving the existence of $\frac{\partial}{\partial
t}\mathscr{F}(t,\omega)$, let us use
$f\in\mathrm{L}^2(\Omega;C[0,T])$, the differentiability of
$\mathfrak{h}$ and choose $|h|<\eta$ to get
\begin{align*}
&\mathbb{E}\left[\frac{\mathscr{F}(t+h,\omega)-\mathscr{F}(t,\omega)}{h}-\frac{\partial}{\partial
t}\mathscr{F}(t,\omega)\right]^2
\nonumber\\
&\leq2\mathbb{E}\left[\int_0^t\left(\frac{\mathfrak{h}(t+h-\tau)-\mathfrak{h}(t-\tau)}{h}-\frac{\partial}{\partial
t}\mathfrak{h}(t-\tau)\right)f(\tau,\omega)\d\tau\right]^2\nonumber\\&
\quad+2\mathbb{E}\left[\frac{1}{h}\int_t^{t+h}\left[\mathfrak{h}(t+h-\tau)f(\tau,\omega)-f(t,\omega)\right]\d\tau\right]^2\nonumber\\
&\leq 2\mathbb{E}\left[\sup_{0\leq\tau\leq
t}|f(\tau,\omega)|^2\right]\int_0^t\left|\frac{\mathfrak{h}(t+h-\tau)-\mathfrak{h}(t-\tau)}{h}-\frac{\partial}{\partial
t}\mathfrak{h}(t-\tau)\right|^2\d
\tau\nonumber\\
&\quad+2\mathbb{E}\left[\frac{1}{h}\int_t^{t+h}\left[\mathfrak{h}(t+h-\tau)f(\tau,\omega)-f(t,\omega)\right]\d\tau\right]^2\nonumber\\
&\leq 2\|f\|^2_{\mathrm{L}^2(\Omega;C[0,T])}\e^2T+
2\mathbb{E}\left[\frac{1}{h}\int_t^{t+h}\left[\mathfrak{h}(t+h-\tau)f(\tau,\omega)-f(t,\omega)\right]\d\tau\right]^2.
\end{align*}
Note that in the time interval $[t, t+h]$ as $h \to 0$, $s \to t$
and hence the function $\mathfrak{h}(t+h-\tau) f(\tau, \omega) \to
 f(t, \omega)$. Thus by the Lebesgue's differentiation
theorem (Theorem 6, Appendix E.4 of Evans \cite{LCE}), the last term
of the right hand side of the above inequality goes to $0$ as $h\to
0$. The arbitrariness of $\e$ gives the required result.
\end{proof}

We have
\begin{align}\label{en21}
\frac{\partial}{\partial
t}\mathscr{F}(t,\omega)=f(t,\omega)+\int_0^t\frac{\partial}{\partial
t}\mathfrak{h}(t-\tau)f(\tau,\omega)\d \tau,
\end{align}
where $$\frac{\partial}{\partial
t}\mathfrak{h}(t)=-8\nu\sum_{n=0}^{\infty}e^{-\nu(2n+1)^2\pi^2t}.$$
From (\ref{en21}), denoting $\frac{\partial}{\partial
t}\mathfrak{h}(t-\tau)$ as $\mathcal{H}(t-\tau)$, we obtain the
Volterra integral equation of the second kind for the determination
of $f(\cdot,\cdot)$ in terms of $\frac{\partial}{\partial
t}\mathscr{F}(t,\omega)$,
\begin{align}\label{en21a}
f(t,\omega)=\frac{\partial}{\partial
t}\mathscr{F}(t,\omega)-\int_0^t\mathcal{H}(t-\tau)f(\tau,\omega)\d
\tau.
\end{align}
\begin{lemma}\label{lemma100}
Let $\frac{\partial}{\partial
t}\mathscr{F}(\cdot,\cdot)\in\mathrm{L}^2(\Omega;\mathrm{L}^2(0,T))$
be given. Then there exists a unique solution
$f(\cdot,\cdot)\in\mathrm{L}^2(\Omega;\mathrm{L}^2(0,T))$ for the
integral equation (\ref{en21a}) in the form
\begin{equation}\label{en21b}
f(t,\omega)=\sum_{n=0}^{\infty}\left[\mathscr{K}^n\frac{\partial\mathscr{F}}{\partial
t}\right](t,\omega),
\end{equation}
where $\mathscr{K}$ is given by
\begin{equation}\label{en21c}
\left[\mathscr{K}\psi\right](t,\omega)=-\int_0^t\mathcal{H}(t-\tau)\psi(\tau,\omega)\d\tau=-(\mathcal{H}*\psi)(t,\omega).
\end{equation}
\end{lemma}
\begin{proof}
Let us begin with the convolution (\ref{en21c}) and apply Young's
inequality for convolutions to obtain
\begin{align}\label{en21d}
\left\|\mathscr{K}\psi(\cdot,\cdot)\right\|^2_{\mathrm{L}^2(\Omega;\mathrm{L}^2(0,T))}
&=\mathbb{E}\left(\int_0^T\left|-\int_0^t\mathcal{H}(t-\tau)\psi(\tau,\omega)\d\tau\right|^2\d
t\right)\nonumber\\&=\mathbb{E}\left(\|(\mathcal{H}*\psi)(\cdot,\cdot)\|_{\mathrm{L}^2(0,T)}^2\right)
\nonumber\\&\leq
\mathbb{E}\left(\|(\mathcal{H}(\cdot)\|^2_{\mathrm{L}^1(0,T)}\|\psi(\cdot,\cdot)\|_{\mathrm{L}^2(0,T)}^2\right)
\nonumber\\&
=\left\|\mathcal{H}(\cdot)\right\|^2_{\mathrm{L}^1(0,T)}\mathbb{E}\left(\left\|\psi(\cdot,\cdot)\right\|_{\mathrm{L}^2(0,T)}^2\right)\nonumber\\
&=\left\|\mathcal{H}(\cdot)\right\|^2_{\mathrm{L}^1(0,T)}\left\|\psi(\cdot,\cdot)\right\|^2_{\mathrm{L}^2(\Omega;\mathrm{L}^2(0,T))}
\end{align}
Let us define
$$\rho=\left\|\mathcal{H}(\cdot)\right\|_{\mathrm{L}^1(0,T)}\leq
\mathfrak{h}(0)-\mathfrak{h}(T)=1-\mathfrak{h}(T)<1.$$ Hence from
(\ref{en21d}), we get \begin{align}\label{en21e}
\left\|\mathscr{K}\psi(\cdot,\cdot)\right\|_{\mathrm{L}^2(\Omega;\mathrm{L}^2(0,T))}
<
\rho\left\|\psi\right\|_{\mathrm{L}^2(\Omega;\mathrm{L}^2(0,T))},\;\forall\;\psi\in\mathrm{L}^2(\Omega;\mathrm{L}^2(0,T)).
\end{align}
Let us define the map $\mathscr{J}$ to be
\begin{align}
\left[\mathscr{J}f\right](t,\omega)=\frac{\partial}{\partial
t}\mathscr{F}(t,\omega)+\left[\mathscr{K}f\right](t,\omega).
\end{align}
Let us denote the approximate solutions of (\ref{en21a}) as
$f_0,f_1,\cdots,f_n$ and they are given by
\begin{align}\label{en21f}
f_0(t,\omega)&=\frac{\partial}{\partial t}\mathscr{F}(t,\omega),\nonumber\\
f_1(t,\omega)&= \frac{\partial}{\partial
t}\mathscr{F}(t,\omega)-\int_0^t\mathcal{H}(t-\tau)f_0(\tau,\omega)\d\tau\nonumber\\&=\frac{\partial}{\partial
t}\mathscr{F}(t,\omega)+\left[\mathcal{K}f_0\right](t,\omega)
=\left[\mathscr{J}f_0\right](t,\omega)=\sum_{j=0}^1\left[\mathscr{K}^j\frac{\partial\mathscr{F}}{\partial
t}\right](t,\omega),\nonumber\\ f_2(t,\omega)&=
\frac{\partial}{\partial
t}\mathscr{F}(t,\omega)-\int_0^t\mathcal{H}(t-\tau)f_1(\tau,\omega)\d\tau\nonumber\\&=\frac{\partial}{\partial
t}\mathscr{F}(t,\omega)+\left[\mathcal{K}f_0\right](t,\omega)+\left[\mathcal{K}^2f_0\right](t,\omega)
\nonumber\\&=\left[\mathscr{J}^2f_0\right](t,\omega)=\sum_{j=0}^2\left[\mathscr{K}^j\frac{\partial\mathscr{F}}{\partial
t}\right](t,\omega),\nonumber\\& \quad\vdots \nonumber\\
f_n(t,\omega)&= \frac{\partial}{\partial
t}\mathscr{F}(t,\omega)+\left[\mathcal{K}f_{n-1}\right](t,\omega)=\left[\mathscr{J}f_{n-1}\right](t,\omega).
\end{align}
Clearly these approximations satisfy
\begin{align}\label{en21g}
f_n(t,\omega)=\left[\mathscr{J}^nf_0\right](t,\omega)=\sum_{j=0}^n\left[\mathscr{K}^j\frac{\partial\mathscr{F}}{\partial
t}\right](t,\omega)
\end{align}
and
\begin{align}\label{en21h}
f_n(t,\omega)-f_{n-1}(t,\omega)=\left[\mathscr{K}^n\frac{\partial\mathscr{F}}{\partial
t}\right](t,\omega).
\end{align}
Hence by using the estimate (\ref{en21e}) on (\ref{en21h}), we
obtain
\begin{align}\label{en21i}
\left\|f_n(\cdot,\cdot)-f_{n-1}(\cdot,\cdot)\right\|_{\mathrm{L}^2(\Omega;\mathrm{L}^2(0,T))}<
\rho^n\left\|\frac{\partial\mathscr{F}}{\partial
t}\right\|_{\mathrm{L}^2(\Omega;\mathrm{L}^2(0,T))}\textrm{ with
}\rho<1.
\end{align}
The existence of a unique fixed point
$f(t,\omega)\in\mathrm{L}^2(\Omega;\mathrm{L}^2(0,T))$ to the map
$\mathscr{J}$ is ensured by Banach contraction mapping theorem and
using the estimate (\ref{en21i}). The unique fixed point
$f(t,\omega)\in\mathrm{L}^2(\Omega;\mathrm{L}^2(0,T))$ to the map
$\mathscr{J}$ is such that
\begin{align}\label{en21j}
\lim_{n\to\infty}\left\|f(\cdot,\cdot)-f_n(\cdot,\cdot)\right\|_{\mathrm{L}^2(\Omega;\mathrm{L}^2(0,T))}=0
\end{align}
and
\begin{align}\label{en21k}
\left[\mathscr{J}f\right](t,\omega)=f(t,\omega)=\sum_{n=0}^{\infty}\left[\mathscr{K}^n\frac{\partial\mathscr{F}}{\partial
t}\right](t,\omega).
\end{align}
\end{proof}
Let us now go back to the system (\ref{eqn1})-(\ref{eqn2}) and note
that a given $\frac{\partial}{\partial
t}\mathscr{F}(\cdot,\cdot)\in\mathrm{L}^2(\Omega;\mathrm{L}^2(0,T))$
corresponds uniquely to a solution $w_1(\cdot,\cdot,\cdot)$ given by
(\ref{eqn3}) and (\ref{en21b}). Moreover, $P(\cdot,\cdot,\cdot)$ is
given by
$$P(x,t,\omega)=-xf(t,\omega)+g_1(t,\omega),$$ for some arbitrary scalar function
$g_1(t,\omega)$ (which can be set to zero for a. e. $\Omega$) for
each $\omega\in\Omega$. The regularity of $w_1(\cdot,\cdot,\cdot)$
can be verified by estimating (\ref{eqn3}).

Let us now obtain a stochastic version of the Theorem 2 (section 3)
from Sritharan \cite{SS4}.

\begin{theorem}
Let $\frac{\partial}{\partial t}\mathscr{F}(\cdot,\cdot)\in
\mathrm{L}^2(\Omega;\mathrm{L}^2(0,T))$ be given. Then the solution
$(w_1,P)$ of (\ref{eqn1})-(\ref{eqn2}) is unique and for $ 0\leq
\e<\frac{1}{2}$
$$w_1\in \mathrm{L}^2(\Omega;\mathrm{L}^2(0,T;H_0^1(0,1)\cap H^{2+\e}(0,1))\cap
C([0,T];H_0^1(0,1)))$$ and $$\partial_t w_1\in
\mathrm{L}^2(\Omega;\mathrm{L}^2(0,T;H^\e(0,1))).$$
\end{theorem}
\begin{proof}
We only need to verify the regularity. Let us denote by
$\mathscr{A}$ the Friedrich's extension of the operator
$-\partial_y^2$ to $H_0^1(0,1)\cap H^2(0,1).$ Then from
\eqref{eqn3}, by using the uniform convergence of the series
solution $w_1(y,t,\omega)$, Fubini's  theorem and Young's inequality
for convolutions, we have
\begin{align}\label{eqn15}
&\mathbb{E}\left(\int_0^T\|\mathscr{A}^{1+\e/2}w_1(y,t,\omega)\|_{\mathrm{L}^2(0,1)}^2\d
t\right)\nonumber\\
&=\mathbb{E}\left(\int_0^T\left\|\frac{4}{\pi}\sum_{n=0}^\infty\left\{\int_0^tf(\tau,\omega)
\frac{1}{(2n+1)}e^{-\nu(2n+1)^2\pi^2(t-\tau)}\d
\tau\right\}\right.\right.\nonumber\\&\qquad\qquad\left.\left.[\sin((2n+1)\pi
y)(2n+1)^2\pi^2]^{(1+\e/2)}\right\|^2_{\mathrm{L^2}(0,1)}\d
t\right)\nonumber\\&\leq C(\pi)\mathbb{E}\left(\sum_{n=0}^\infty
(2n+1)^{2+2\e}\int_0^T\left|\int_0^tf(\tau,\omega)e^{-\nu(2n+1)^2\pi^2(t-\tau)}\d
\tau\right|^2\d t\right)\nonumber\\&\leq
C(\pi)\mathbb{E}\left(\sum_{n=0}^\infty
(2n+1)^{2+2\e}\left\|\int_0^tf(\tau,\omega)e^{-\nu(2n+1)^2\pi^2(t-\tau)}\d
\tau\right\|^2_{\mathrm{L}^2(0,T)}\right)\nonumber\\&\leq
C(\pi)\mathbb{E}\left(\sum_{n=0}^\infty
(2n+1)^{2+2\e}\|f(\cdot,\cdot)\|^2_{\mathrm{L}^2(0,T)}\|e^{-\nu(2n+1)^2\pi^2t}\|_{\mathrm{L}^1(0,T)}^2\right)\nonumber\\&
\leq
C(\pi)\mathbb{E}\left(\|f(\cdot,\cdot)\|^2_{\mathrm{L}^2(0,T)}\right)\sum_{n=0}^\infty
(2n+1)^{2+2\e}\left(\frac{1}{\nu(2n+1)^2\pi^2}\right)^2\nonumber\\&
\leq
C_1(\pi,T,\nu)\|{f}(\cdot,\cdot)\|^2_{\mathrm{L}^2(\Omega;\mathrm{L}^2(0,T))}\sum_{n=0}^\infty
(2n+1)^{-2+2\e}.
\end{align}
Thus noting that $\sum_{n=0}^{\infty}(2n+1)^{-2+2\e}<\infty$ for
$\e<\frac{1}{2},$ we obtain
\begin{align}\label{eq1000}
\mathbb{E}\left(\int_0^T\|\mathscr{A}^{1+\e/2}w_1\|_{\mathrm{L}^2(0,1)}^2\d
t\right)<\infty. \end{align} Then one can deduce that,
$\mathscr{A}^{1+\e/2}w_1\in
\mathrm{L}^2(\Omega;\mathrm{L}^2(0,T;\mathrm{L}^2(0,1)))$ which
implies $w_1\in
\mathrm{L}^2(\Omega;\mathrm{L}^2(0,T;\mathrm{H}^{2+\e}(0,1))).$
Since $\mathscr{A}$ is the Friedrich's extension of the operator
$-\partial^2_y$ to $H_0^1(0,1)\cap H^2(0,1)$ and from
(\ref{eq1000}), we get
$$\mathbb{E}\left(\int_0^T\|\mathscr{A}^{\e/2}\partial_tw_1\|^2_{\mathrm{L}^2(0,1)}\d
t\right)<\infty,$$ which gives $\partial_t w_1\in
\mathrm{L}^2(\Omega;\mathrm{L}^2(0,T;H^\e(0,1))) \textrm{ for }
0\leq \e<\frac{1}{2}.$
\end{proof}

\begin{figure}[h]
\begin{center}
\includegraphics[width=.5\textwidth]{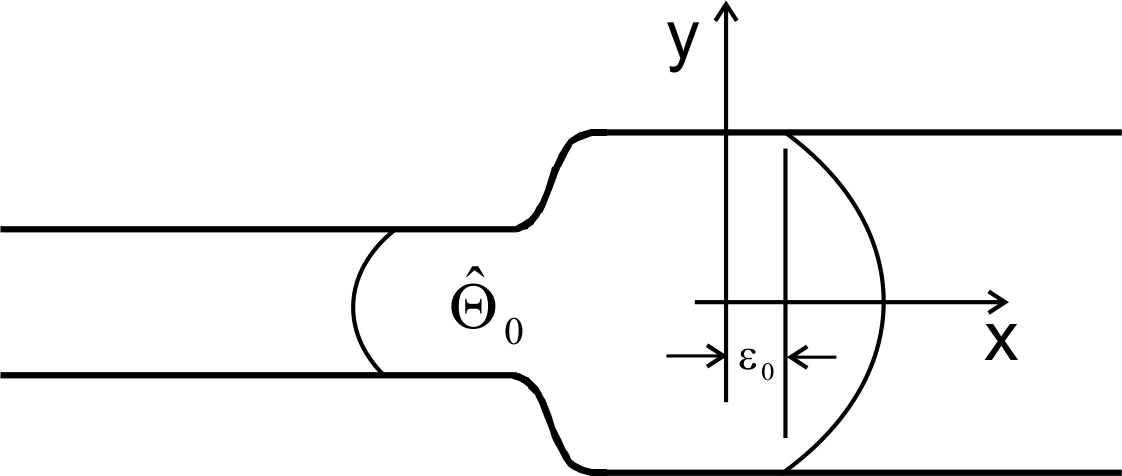}
\end{center}
\caption{The domain $\hat{\Theta}_0$.} \label{fig3}
\end{figure}

Let us now describe a method to extend the constructed flux
$\mathscr{F}(\cdot,\cdot)$ carried in $\mathbf{O}_i$ into the domain
$\Theta_0$ in a smooth manner. Let $\hat{\Theta}_0$ be an open
subset of $\Theta$ with compact closure such that
$\Theta_0\subset\hat{\Theta}_0\Subset\Theta$. Moreover,
$\partial\hat{\Theta}_0$ is of class $C^{\infty}$ (see Figure
\ref{fig3}).
\begin{proposition}\label{prop1}
There exists a vector field
$\hat{\w}:\hat{\Theta}_0\times[0,T]\times\Omega\rightarrow\mathbb{R}^2$
in the form
$\hat{\w}(\x,t,\omega)=(\hat{\psi}_y,-\hat{\psi}_x)(\x,t,\omega)$
for some stream function $\hat{\psi}$ such that
$$\nabla\cdot\hat{\w}=0\;\textrm{ in }\;\hat{\Theta}_0\times(0,T)\times\Omega$$ in
the sense of distributions and
\begin{align*}
\hat{\w}(\x,t,\omega)&=0,\;
t\in(0,T),\;\x=(x,y)\in\partial\hat{\Theta}_0\cap\partial\Theta,\\
\hat{\w}(\x,t,\omega)&=(w_1^{(i)}(y,t,\omega),0),\;t\in(0,T),\;\x=(x,y)\in\partial\hat{\Theta}_0\cap\mathbf{O}_i,\;i=1,2,
\end{align*}
for almost all $\omega\in\Omega$, in the sense of trace. Here
$w^{(i)}_1(\cdot,\cdot,\cdot)$ denotes the solution of the heat
equation constructed in $\mathbf{O}_i$. Moreover,
$$\hat{\w}\in\mathrm{L}^2(\Omega;\mathrm{L}^2(0,T;H^2(\hat{\Theta}_0))\cap
C([0,T];H^1(\hat{\Theta}_0)))$$ with $\partial_t\hat{\w}\in
\mathrm{L}^2(\Omega;\mathrm{L}^2(0,T;\mathrm{L}^2(\hat{\Theta}_0)))\subset\mathrm{L}^2(\Omega;\mathrm{L}^2(0,T;H^{-1}(\hat{\Theta}_0))).$
\end{proposition}

Proposition \ref{prop1} can be deduced from the well-known method of
extending the boundary data (with zero net flux) into the domain as
a divergence-free vector field (Ladyzhenskaya \cite{La}). Also we
note that the regularity of the boundary data corresponds to that
obtained in the solution of the heat equation and hence we have

$$\hat{\w}\big|_{\partial\hat{\Theta}_0}\in\mathrm{L}^2(\Omega;\mathrm{L}^2(0,T;H^{\frac{3}{2}+\e}(\partial\hat{\Theta}_0))\cap
C([0,T];H^{\frac{1}{2}}(\partial\hat{\Theta}_0)))$$ and
$$\partial_t\hat{\w}\big|_{\partial\hat{\Theta}_0}\in\mathrm{L}^2(\Omega;\mathrm{L}^2(0,T;H^{\e-\frac{1}{2}}(\partial\hat{\Theta}_0))),\;0\leq
\e<\frac{1}{2}.$$ Let us use $\hat{\w}(\x,t,\omega)$ in
$\hat{\Theta}_0$ and $w_1^{(i)}(y,t,\omega)$ in $\mathbf{O}_i$,
$i=1,2$ to construct the desired vector field $\w(\x,t,\omega)$ in
$\Theta$ (see Sritharan \cite{SS4} for more details about the
construction).

Let $\psi^{(i)}(y,t,\omega)$ be the stream function corresponding to
the vector field
$$(w_1^{(i)}(y,t,\omega),0)=(\partial_y\psi^{(i)}(y,t,\omega),\partial_x\psi^{(i)}(y,t,\omega))$$
in $\mathbf{O}_i$. Let us set $\psi^{(i)}\big|_{y=0}=0$ on the lower
wall of $\mathbf{O}_i$, hence on the upper wall of $\mathbf{O}_i$,
we obtain
\begin{align*}
\mathscr{F}(t,\omega)=\int_0^1w_1(y,t,\omega)\d
y=\int_0^1\frac{\partial}{\partial y}\psi^{(i)}(y,t,\omega)\d
y=\psi^{(i)}(y,t,\omega)\Big|_{0}^1=\psi^{(i)}(1,t,\omega).
\end{align*}
Also due to the flux condition, we have
$$\hat{\psi}\big|_{\partial\Theta_2\cap\partial\hat{\Theta}_0}=\mathscr{F}(t,\omega)\textrm{
and }
\hat{\psi}\big|_{\partial\Theta_1\cap\partial\hat{\Theta}_0}=0.$$

Let us construct a $C^{\infty}$ function
$\lambda(\cdot):\mathbb{R}\to(0,1)$ by mollifying a step function
such that
$$\lambda(\xi)=\left\{\begin{array}{cl}0&\textrm{ for }\xi\leq 0,
\\ 1&\textrm{ for }0<\e_0\leq\xi,\end{array}\right.$$ where $\e_0$
is a small number. Let us define
$$\psi(x,y,t,\omega)=(1-\lambda(x))\hat{\psi}(x,y,t,\omega)+\lambda(x)\psi^{(i)}(y,t,\omega)\textrm{
in } \hat{\Theta}_0\cap\mathbf{O}_i,i=1,2,$$ for each
$\omega\in\Omega$.

In this way we obtain a stream function which takes
$\hat{\psi}(x,y,t,\omega)$ in $\Theta_0$ and smoothly become
$\psi^{(i)}(y,t,\omega)$ in $\mathbf{O}_i$. On the lower wall
$\partial\Theta_1$, we have
$$\psi(\x,t,\omega)=\hat{\psi}(\x,t,\omega)=\psi^{(i)}(y,t,\omega)=0$$
and on the upper wall $\partial\Theta_2,$ we have
$$\psi(\x,t,\omega)=\hat{\psi}(\x,t,\omega)=\psi^{(i)}(y,t,\omega)=\mathscr{F}(t,\omega).$$
Hence we obtain the desired divergence free basic vector field as
$$\w(\x,t,\omega)=(\partial_y\psi,-\partial_x\psi)(\x,t,\omega).$$

A similar extension can be constructed for the scalar field
$P(\x,t,\omega)$ so that in $\Theta_0$,
$$P(\x,t,\omega)\in\mathrm{L}^2(\Omega;\mathrm{L}^2(0,T;H^1(\Theta_0)))$$
and becomes $-xf(t,\omega)$ smoothly in $\mathbf{O}_i$.

Let us now note some of the properties of the constructed basic
vector field  $\w(\x,t,\omega)$:
\begin{itemize}
\item [(i)] The flux condition, $\mathscr{F}(t,\omega)=\int_{\Gamma}\w(\x,t,\omega)\cdot\n\d S$, where $\Gamma$
is any cross section of the channel.
\item [(ii)] On the boundary, $\x\in\partial\Theta$, $\w(\x,t,\omega)=0$,
for all $t\in[0,T]$ and $\omega\in\Omega$.
\item [(iii)] In the bounded region $\Theta_0$, the basic vector field $\w$ satisfies $$\w\in
\mathrm{L}^2(\Omega;\mathrm{L}^2(0,T;H^2(\Theta_0))\cap
C([0,T];H^1(\Theta_0)))$$ with $\w_t\in
\mathrm{L}^2(\Omega;\mathrm{L}^2(0,T;H^{-1}(\Theta_0))).$ This
gives, in $\Theta_0$
\begin{align}\label{eqn1001}
\w_t-\nu\Delta\w+\w\cdot\nabla\w\in\mathrm{L}^2(\Omega;\mathrm{L}^2(0,T;H^{-1}(\Theta_0))).
\end{align}
By the regularity of $\w$ in $\Theta_0$, we have
\begin{align}\label{eqn1001a}\sup_{\x\in\Theta_0}|\w(\x,\cdot,\omega)|=
\beta_{10}(\cdot,\omega),\;\;
\|\nabla\w(\x,t,\omega)\|_{\mathrm{L}^2(0,T;\mathrm{L}^2(\Theta_0))}=
\beta_{20}(\omega),\end{align} where
$\beta_{10}(\cdot,\cdot)\in\mathrm{L}^2(\Omega;\mathrm{L}^2(0,T))$
and $\beta_{20}(\cdot)\in\mathrm{L}^2(\Omega)$.
\item [(iv)] In the infinite strips $\mathbf{O}_i$, (actually in
$\mathbf{O}_i\backslash\hat{\Theta}_0$) $\w(x,y,t,\omega)$ becomes
$w^{(i)}(y,t,\omega)$ and hence satisfies properties (i), (ii)
above.

Since, $w_1^{(i)}\in \mathrm{L}^2(\Omega;C^1([0,T];H^1(0,1))),$ we
have
\begin{align}\label{eqn1002}
\sup_{(\x,t)\in\mathbf{O}_i\times(0,T)}|\w(\x,t,\omega)|=\sup_{(y,t)\in(0,1)\times(0,T)}|w_1^{(i)}(y,t,\omega)|=
\beta_{1i}(\omega),
\end{align}
with $\beta_{1i}(\cdot)\in\mathrm{L}^2(\Omega).$

Also, since $w_1^{(i)}\in
\mathrm{L}^2(\Omega;\mathrm{L}^2(0,T;H^2(0,1)))$, we have
\begin{align}\label{eqn1003}
\sup_{\x\in\mathbf{O}_i}|\nabla\w(\x,\cdot,\omega)|=\sup_{y\in(0,1)}|\partial_y
w_1^{(i)}(y,\cdot,\omega)|= \beta_{2i}(\cdot,\omega),
\end{align}
with
$\beta_{2i}(\cdot,\cdot)\in\mathrm{L}^2(\Omega;\mathrm{L}^2(0,T))$.

In $\mathbf{O}_i$, $w_1^{(i)}$ is independent of the axial direction
$x$, we have the constructed vector field $\w(\cdot,\cdot,\cdot)$
(for almost all time and almost all $\omega\in\Omega$) has infinite
energy $\|\w\|_{\mathrm{L}^2(\Theta)}$ and infinite dissipation
$\|\nabla\w\|_{\mathrm{L}^2(\Theta)}$ in $\Theta$, since
$$\int_0^{\infty}\int_0^1|w_1^{(i)}(y,t,\omega)|^2\d y\d x=\infty\textrm{ and
}\int_0^{\infty}\int_0^1|\partial_yw_1^{(i)}(y,t,\omega)|^2\d y\d
x=\infty,$$ $\mathbb{P}$ - a. s.
\item [(v)] Also note that in $\mathbf{O}_i$ (actually in
$\mathbf{O}_i\backslash\hat{\Theta}_0$),
$$\w_t-\nu\Delta\w+\w\cdot\nabla\w+\nabla P=0,$$ for almost all
$\omega\in\Omega.$
\end{itemize}

\subsection{Connection between $\mathscr{F}$ and $\beta$.}
Now let us establish certain connection between the norm of
$\mathscr{F}$ and $\beta$ by using the properties of $\w$ in
$\Theta_0$ and in the two outlets $\mathbf{O}_1$ and $\mathbf{O}_2$.
\begin{proposition}\label{prop11}
In $\Theta=\mathbf{O}_1\cup\mathbf{O}_2\cup\Theta_0$, the regularity
properties of the basic vector field $\w$ implies that
$\beta_{10}(\cdot,\cdot)\in\mathrm{L}^2(\Omega;\mathrm{L}^2(0,T));\;\beta_{11}(\cdot),\beta_{12}(\cdot)\in\mathrm{L}^2(\Omega)$
and
$\beta_{20}(\cdot)\in\mathrm{L}^2(\Omega);\;\beta_{21}(\cdot,\cdot),\beta_{22}(\cdot,\cdot)\in\mathrm{L}^2(\Omega;\mathrm{L}^2(0,T)).$
Thus in $\Theta$, we have
\begin{align}\label{eqn1004}
\mathbb{E}\left[\int_0^T|\beta_{10}(t,\omega)|^2\d
t+\left(\beta^2_{11}(\omega)+\beta^2_{12}(\omega)\right)T\right]\leq
C(\pi,\nu,T)\|\mathscr{F}\|_{_{\mathrm{L}^2(\Omega;H^1(0,T))}}^2
\end{align}
 and
\begin{align}\label{eqn1005}
\mathbb{E}\left[\beta^2_{20}(\omega)T+\int_0^T|\beta_{21}(t,\omega)|^2\d
t+\int_0^T|\beta_{22}(t,\omega)|^2\d t\right]\leq
C(\pi,\nu,T)\|\mathscr{F}\|_{_{\mathrm{L}^2(\Omega;H^1(0,T))}}^2,
\end{align}
where $\beta$'s are defined in (\ref{eqn1001a}), (\ref{eqn1002}) and
(\ref{eqn1003}).
\end{proposition}
\begin{proof}
If we take $1+\e/2=\varsigma$ in the inequality \eqref{eqn15}, one
can see that for $\rho<1$
\begin{align}\label{eqn16}
&\mathbb{E}\left(\int_0^T\|\mathscr{A}^{\varsigma}w_1(y,t,\omega)\|^2_{\mathrm{L}^2(0,1)}\d
t\right)\nonumber\\&\leq
C_1(\pi,T,\nu)\|f(\cdot,\cdot)\|^2_{\mathrm{L}^2(\Omega;\mathrm{L}^2(0,T))}\sum_{n=0}^\infty(2n+1)^{4\varsigma-6}
\nonumber\\&\leq
C_1(\pi,T,\nu)\left\|\sum_{p=0}^{\infty}\left[\mathscr{K}^p\frac{\partial}{\partial
t}\mathscr{F}(t,\omega)\right]\right\|^2_{\mathrm{L}^2(\Omega;\mathrm{L}^2(0,T))}\sum_{n=0}^\infty(2n+1)^{4\varsigma-6}.
\end{align}
Now let us take the term
$\left\|\sum_{p=0}^{\infty}\left[\mathscr{K}^p\frac{\partial}{\partial
t}\mathscr{F}(t,\omega)\right]\right\|^2_{\mathrm{L}^2(\Omega;\mathrm{L}^2(0,T))}$
and use (\ref{en21a}), (\ref{en21i}) and Cauchy-Schwarz inequality
to get
\begin{align*}
&\left\|\sum_{p=0}^{\infty}\left[\mathscr{K}^p\frac{\partial}{\partial
t}\mathscr{F}(t,\omega)\right]\right\|^2_{\mathrm{L}^2(\Omega;\mathrm{L}^2(0,T))}\nonumber\\&
=
\mathbb{E}\left(\left\|\sum_{p=0}^{\infty}\left[\mathscr{K}^p\frac{\partial}{\partial
t}\mathscr{F}(t,\omega)\right]\right\|^2_{\mathrm{L}^2(0,T)}\right)\nonumber\\&
=
\mathbb{E}\left(\left<\sum_{p=0}^{\infty}\left[\mathscr{K}^p\frac{\partial}{\partial
t}\mathscr{F}(t,\omega)\right],\sum_{m=0}^{\infty}\left[\mathscr{K}^m\frac{\partial}{\partial
t}\mathscr{F}(t,\omega)\right]\right>_{\mathrm{L}^2(0,T)}\right)\nonumber\\&
\leq
\mathbb{E}\left(\sum_{p=0}^{\infty}\sum_{m=0}^{\infty}\left|\left<\left[\mathscr{K}^p\frac{\partial}{\partial
t}\mathscr{F}(t,\omega)\right],\left[\mathscr{K}^m\frac{\partial}{\partial
t}\mathscr{F}(t,\omega)\right]\right>_{\mathrm{L}^2(0,T)}\right|\right)\nonumber\\&
\leq
\mathbb{E}\left(\sum_{p=0}^{\infty}\sum_{m=0}^{\infty}\left\|\mathscr{K}^p\frac{\partial}{\partial
t}\mathscr{F}(t,\omega)\right\|_{\mathrm{L}^2(0,T)}\left\|\mathscr{K}^m\frac{\partial}{\partial
t}\mathscr{F}(t,\omega)\right\|_{\mathrm{L}^2(0,T)}\right)\nonumber\\&
<
\mathbb{E}\left(\sum_{p=0}^{\infty}\sum_{m=0}^{\infty}\rho^p\left\|\frac{\partial}{\partial
t}\mathscr{F}(t,\omega)\right\|_{\mathrm{L}^2(0,T)}\rho^m\left\|\frac{\partial}{\partial
t}\mathscr{F}(t,\omega)\right\|_{\mathrm{L}^2(0,T)}\right)\nonumber\\&
=\mathbb{E}\left(\sum_{p=0}^{\infty}\rho^p\sum_{m=0}^{\infty}\rho^m\left\|\frac{\partial}{\partial
t}\mathscr{F}(t,\omega)\right\|_{\mathrm{L}^2(0,T)}^2\right)=
\mathbb{E}\left(\left(\frac{1}{1-\rho}\right)^2\left\|\frac{\partial\mathscr{F}}{\partial
t}\right\|_{\mathrm{L}^2(0,T)}^2\right)\nonumber\\&
=\left(\frac{1}{1-\rho}\right)^2\left\|\frac{\partial\mathscr{F}}{\partial
t}\right\|_{\mathrm{L}^2(\Omega;\mathrm{L}^2(0,T))}^2
\end{align*}
Hence the inequality (\ref{eqn16}) becomes
\begin{align*}
&\mathbb{E}\left(\int_0^T\|\mathscr{A}^{\varsigma}w_1(y,t,\omega)\|^2_{\mathrm{L}^2(0,1)}\d
t\right) \nonumber\\&\leq
C_1(\pi,T,\nu)\left(\frac{1}{1-\rho}\right)^2\left\|\frac{\partial\mathscr{F}}{\partial
t}\right\|_{\mathrm{L}^2(\Omega;\mathrm{L}^2(0,T))}^2\sum_{n=0}^\infty(2n+1)^{4\varsigma-6}.
\end{align*}
The series $\sum_{n=0}^\infty(2n+1)^{4\varsigma-6}$ is convergent
for $0\leq \varsigma<\frac{5}{4}$. Hence for $0\leq
\varsigma<\frac{5}{4}$ and for $\rho<1$, we get
\begin{align}\label{eqn16a}\mathbb{E}\left(\int_0^T\|\mathscr{A}^{\varsigma}w_1(y,t,\omega)\|^2_{\mathrm{L}^2(0,1)}\d
t\right)&\leq
C_2(\pi,T,\nu)\sum_{n=0}^\infty(2n+1)^{4\varsigma-6}\left\|\frac{\partial\mathscr{F}}{\partial
t}\right\|_{\mathrm{L}^2(\Omega;\mathrm{L}^2(0,T))}^2\nonumber\\&\leq
C_3\|\mathscr{F}(t)\|^2_{_{\mathrm{L}^2(\Omega;H^1(0,T))}}.\end{align}
By using the estimates (\ref{eqn1002}) and (\ref{eqn1003}), by
taking $\varsigma=0,\frac{1}{2}$  in (\ref{eqn16a}), one can easily
see that
\begin{align*}&\mathbb{E}\left[\int_0^T|\beta_{21}(t,\omega)|^2\d
t+\int_0^T|\beta_{22}(t,\omega)|^2\d t\right]\\&=
\mathbb{E}\left[\int_0^T\sup_{y\in(0,1)}|\partial_yw_1^{(1)}(y,t,\omega)|^2\d
t+\int_0^T\sup_{y\in(0,1)}|\partial_yw_1^{(2)}(y,t,\omega)|^2\d
t\right]\\&\leq
C\|\mathscr{F}(t)\|^2_{_{\mathrm{L}^2(\Omega;H^1(0,T))}}\end{align*}and
\begin{align*}&\mathbb{E}\left[(\beta^2_{11}(\omega)+\beta^2_{12}(\omega))\right]T
\\&= \mathbb{E}\left[\int_0^T\sup_{(y,t)\in(0,1)\times(0,T)}|w_1^{(1)}(y,t,\omega)|^2\d t
+\int_0^T\sup_{(y,t)\in(0,1)\times(0,T)}|w_1^{(2)}(y,t,\omega)|^2\d
t\right]\\&\leq
C\|\mathscr{F}\|^2_{_{\mathrm{L}^2(\Omega;H^1(0,T))}}\end{align*} in
$\mathbf{O}_i$.

Since $\Theta_0$ is bounded, by using property (iii) of the basic
vector field and we extended the constructed flux
$\mathscr{F}(t,\omega)=\int_{\Gamma}\mathbf{w}\cdot\mathbf{n}\d S$
carried in $\mathbf{O}_i$ into the domain $\Theta_0$ in a smooth
manner, it can be easily shown that
$$\mathbb{E}\left[\beta^2_{20}(\omega)\right]T\leq C\|\mathscr{F}\|^2_{_{\mathrm{L}^2(\Omega;H^1(0,T))}} \textrm{ and }
\mathbb{E}\left[\int_0^T|\beta_{10}(t,\omega)|^2\d t\right] \leq
C\|\mathscr{F}\|^2_{_{\mathrm{L}^2(\Omega;H^1(0,T))}}.$$ Hence, we
have the relation (\ref{eqn1004}) and (\ref{eqn1005}) in $\Theta$.
\end{proof}

\begin{remark}\label{rem11}
From (\ref{eqn1003}), we have
$$\sup_{\x\in\mathbf{O}_i}|\nabla\w(\x,\cdot,\omega)|= \beta_{2i}(\cdot,\omega)\in
\mathrm{L}^2(\Omega;\mathrm{L}^2(0,T)).$$ For $i=1,2$, we can always
get an upper bound for $\beta_{2i}(t,\omega)$ for all $t\in[0,T]$
and almost all $\omega\in\Omega$. By using (\ref{eqn3}) and Young's
inequality for convolution of two functions, we obtain
\begin{align}\label{qeta}
&\sup_{0\leq t\leq
T}\sup_{\x\in\mathbf{O}_2}|\nabla\w(\x,t,\omega)|^2\nonumber\\&
=\sup_{0\leq t\leq T}\sup_{y\in(0,1)}\left|\frac{\partial}{\partial
y}\left(\frac{4}{\pi}\sum_{n=0}^{\infty}\left\{\int_0^tf(s,\omega)\frac{1}{2n+1}e^{-\nu(2n+1)^2\pi^2(t-s)}\d
s\right\}\right.\right.\nonumber\\&\quad\sin(2n+1)\pi
y\Big)\Big|^2\nonumber\\&\leq \sup_{0\leq t\leq
T}4\pi\sum_{n=0}^{\infty}\int_0^t|f(s,\omega)|^2e^{-2\nu(2n+1)^2\pi^2(t-s)}\d
s\nonumber\\&\leq 4\pi \sup_{0\leq t\leq
T}\sum_{n=0}^{\infty}\left(\int_0^t|f(s,\omega)|^2\d
s\right)\left(\int_0^te^{-2\nu(2n+1)^2\pi^2t}\d s\right)
\nonumber\\&\leq \frac{2}{\pi \nu}\sup_{0\leq t\leq
T}\sum_{n=0}^{\infty}\frac{1}{(2n+1)^2}\int_0^t|f(s,\omega)|^2\d
s\nonumber\\&\leq
\frac{\pi}{4\nu}\|f(t,\omega)\|^2_{\mathrm{L}^2(0,T)}=\eta_{22}^2(\omega),
\end{align}
where $\eta_{22}(\omega)\in\mathrm{L^2}(\Omega)$, since
$f\in\mathrm{L}^2(\Omega;C([0,T]))\subset\mathrm{L}^2(\Omega;\mathrm{L}^2(0,T))$.

Hence, we get $\sup_{t\in[0,T]}\beta_{2i}^2(t,\omega)\leq
\eta^2_{2i}(\omega)\textrm{ for }i=1,2$, $\mathbb{P}$-a. s.
\end{remark}

\section{The Linear and Multilinear Operators}\setcounter{equation}{0}
In this section we define the Stokes operator, the inertia operator
and the other operators relevant to our analysis. Most of the
results obtained in this section has been taken from \cite{SS4}
given here for completeness.
\subsection{The Stokes Operator}
Let us first define the Stokes operator for the two-dimensional
admissible channel domain and analyze its properties. Let us denote
by $a(\cdot,\cdot)$ the symmetric bilinear form
\begin{align}\label{eqn1010}
a(\u,\v)=\int_{\Theta}\nabla\u\cdot\nabla\v\d\x.
\end{align}
Let us now define the Stokes operator $A$ and its domain $D(A)$ in
the following way. Given $\u\in\mathbf{V}$, if there exists an
element $\mathbf{g}\in\mathbf{H}$ such that
$$a(\u,\v)=(\mathbf{g},\v)_{\mathbf{H}},\
\;\forall\;\v\in\mathbf{V},$$ then we say $\u\in D(A)$ and $A\u=g.$
By Poincar\'{e} inequality $\|\u\|_{\mathrm{L}^2(\Theta)}\leq
C\|\nabla\u\|_{\mathrm{L}^2(\Theta)}$ holds for the admissible
channel domains, we obtain the coerciveness ($\mathbf{V}$-elliptic)
property as,
\begin{align*}
a(\u,\u)=\|\nabla\u\|^2_{\mathrm{L}^2(\Theta)}\geq
C\|\u\|^2,\;\forall\; \u\in\mathbf{V}.
\end{align*}
Since the form $a(\cdot,\cdot)$ is symmetric, continuous, and
positive definite, we have the following standard results obtained
in \cite{Th}.
\begin{proposition}
There exists a self-adjoint, regularly accretive onto map
$\tilde{A}\in\mathscr{L}(D(\tilde{A});\mathbf{H})$ such that
$$a(\u,\v)=(\tilde{A}\u,\v)_{\mathbf{H}},\;\forall\;\u\in
D(A),\;\forall\;\v\in\mathbf{V},$$ with
$D(\tilde{A})\subset\mathbf{V}$ dense. Moreover, it is possible to
extend this operator as an isomorphic onto map
$\hat{A}\in\mathscr{L}(\mathbf{V},\mathbf{V}')$ such that
$$a(\u,\v)=(\hat{A}\u,\v)_{\mathbf{V}\times\mathbf{V}'},\;\forall\;\u,\v\in\mathbf{V}.$$
\end{proposition}
The operators $\tilde{A}$ and $\hat{A}$ are different Friedrich's
extensions of the classical Stokes operator. In the remainder of the
paper, both of these operators will be denoted by $A$. It follows
from a general theorem of Lions \cite{LJ} that
$\mathbf{V}=D(A^{1/2})$. Let the orthogonal Helmhotz-Hodge
projection
$P_{\mathbf{H}}:\mathrm{L}^2(\Theta,\mathbb{R}^2)\rightarrow
\mathrm{H}$ and $A\u=-P_{\mathbf{H}}\Delta \u\;\forall\;\u\in D(A).$

Let us now state a regularity theorem. This give us an explicit
representation of $D(A)$.
\begin{lemma}
Let $\Theta$ be an admissible channel domain. Then, for a given
$\mathbf{g}\in R(A)=\mathbf{H}$, the Stokes problem of finding
$(\mathbf{u},p):\Theta\rightarrow\mathbb{R}^2\times\mathbb{R}$ such
that
\begin{align*}
-\Delta\u+\nabla p&=\mathbf{g}\;\textrm{ in }\;\Theta\\
\nabla\cdot\u&=0\;\textrm{ in }\;\Theta\\
\u\big|_{\partial\Theta}&=0,\;\;\u\rightarrow0\;\textrm{ as
}\;|x|\rightarrow\infty\;\textrm{ in }\;\mathbf{O}_i \textrm{ and
}\\ \int_{\Gamma}\u\cdot\n\d S&=0,
\end{align*}
has a unique solution $\u$ that satisfies $\u\in
D(A)=H^2(\Theta)\cap\mathbf{V}.$ Moreover, if we define the graph
norm
$$\|\u\|_{D(A)}=\|\u\|_{\mathrm{L}^2(\Theta)}+\|A\u\|_{\mathrm{L}^2(\Theta)},\;\forall\;\u\in
D(A),$$ then there exists $C>0$ such that,
$C\|\u\|_{H^2(\Theta)}\leq \|\u\|_{D(A)}\leq \|\u\|_{H^2(\Theta)}.$
\end{lemma}
\begin{proof}
For proof see Theorem 3, section 4 of \cite{SS4}.
\end{proof}

\subsection{The Bilinear Operator (Inertia Term)}
Let us now define the trilinear form and the bilinear operator and
their properties. The trilinear form is given by,
\begin{align}\label{eqn1014}
b(\u,\v,\phi)=\sum_{i,j=1}^2\int_{\Theta}\u_i\frac{\partial\v_j}{\partial
x_i}\phi_j\d \x.
\end{align}
It is well known that trilinear form
$b(\cdot,\cdot,\cdot):\mathbf{V}\times\mathbf{V}\times\mathbf{V}\rightarrow\mathbb{R}$
is continuous and for all $\u,\v,\phi\in\mathbf{V}$ we have
\begin{align*}
b(\u,\v,\phi)=-b(\u,\phi,\v)\;\textrm{ and }\;b(\u,\v,\v)=0.
\end{align*}
The following lemma is standard.
\begin{lemma}\label{buu} \cite{SS4}
There exist continuous bilinear operators,
$B(\cdot,\cdot):\mathbf{V}\times\mathbf{V}\rightarrow\mathbf{V}'$
and
$B^*(\cdot,\cdot):\mathbf{V}\times\mathbf{V}\rightarrow\mathbf{V}'$
such that
\begin{align}\label{eqn1016}
b(\u,\v,\phi)=(B(\u,\v),\phi)_{\mathbf{V}'\times\mathbf{V}}=(B^*(\v,\phi),\u)_{\mathbf{V}'\times\mathbf{V}},\;\forall\;\u,\v,\phi\in\mathbf{V}.
\end{align}
Moreover,
\begin{align}\label{eqn1017}
\|B(\u,\v)\|_{\mathbf{V}'}&\leq
C\|\u\|^{1/2}_{\mathrm{L}^2(\Theta)}\|\nabla\u\|^{1/2}_{\mathrm{L}^2(\Theta)}\|\v\|^{1/2}_{\mathrm{L}^2(\Theta)}\|\nabla\v\|^{1/2}_{\mathrm{L}^2(\Theta)},
\\ \|B^*(\v,\phi)\|_{\mathbf{V}'}&\leq
C\|\nabla\v\|_{\mathrm{L}^2(\Theta)}\|\phi\|^{1/2}_{\mathrm{L}^(\Theta)}\|\nabla\phi\|_{\mathrm{L}^2(\Theta)}^{1/2}
\end{align}
and
\begin{align}\label{eqn1018}
B^*(\v,\phi)=-B^*(\phi,\v),\quad \forall\;\phi,\v\in\mathbf{V}.
\end{align}
\end{lemma}

Now we will provide a general setting for estimating the terms
$b(\w,\cdot,\cdot)$ and $b(\cdot,\w,\cdot)$, where $\w$ is the basic
vector field constructed in section $3$.
\begin{lemma}\label{bwv}
The trilinear form
$b(\w,\cdot,\cdot):\mathbf{V}\times\mathbf{V}\rightarrow\mathbb{R}$
is continuous and satisfies,
\begin{align}\label{eqn1019}
(i)\;\;
b(\w,\u,\v)&=-b(\w,\v,\u),\;\forall\;\u,\v\in\mathbf{V}\;\textrm{
and }\\
(ii)\;\; b(\w,\u,\u)&=0,\;\forall\;\u\in\mathbf{V}.
\end{align}
\end{lemma}
\begin{proof}
For proof see Lemma 4, section 4 of \cite{SS4}.
\end{proof}
\begin{lemma}\label{bvw1}
The form
$b(\cdot,\w,\cdot):\mathbf{V}\times\mathbf{V}\rightarrow\mathbb{R}$
is continuous.
\end{lemma}
\begin{proof}
For proof see Lemma 5, section 4 of \cite{SS4}.
\end{proof}
\begin{lemma}
There exists a continuous linear operator
$B_1(\w)\in\mathscr{L}(\mathbf{V};\mathbf{V}')$ such that
\begin{align*}
(i)\;\;
b(\w,\u,\v)=(B_1(\w)\u,\v)_{\mathbf{V}'\times\mathbf{V}},\;\forall\;\u,\v\in\mathbf{V}.
\end{align*}
Moreover,
\begin{align*}
(ii)\;\;(B_1(\w)\u,\v)_{\mathbf{V}'\times\mathbf{V}}=-(B_1(\w)\v,\u)_{\mathbf{V}'\times\mathbf{V}},\;\forall\;\u,\v\in\mathbf{V}.
\end{align*}
\end{lemma}
\begin{proof}
For proof see Lemma 6, section 4 of \cite{SS4}.
\end{proof}
\begin{lemma}\label{bvw}
There exists a continuous linear operators
$B_2(\w)\in\mathscr{L}(\mathbf{V};\mathbf{V}')$ and
$B_2(\w)^*\in\mathscr{L}(\mathbf{V}';\mathbf{V})$ such that
\begin{align*}
(i)\;\;(B_2(\w)\u,\v)_{\mathbf{V}'\times\mathbf{V}}=(\u,B_2(\w)^*\v)_{\mathbf{V}'\times\mathbf{V}}
=b(\u,\w,\v),\;\forall\;\u,\v\in\mathbf{V}.
\end{align*}
Moreover,
\begin{align*}
(ii)\;\;\|B_2(\w)\|_{\mathscr{L}(\mathbf{V};\mathbf{V'})}=\|B_2(\w)^*\|_{\mathscr{L}(\mathbf{V'};\mathbf{V})}\leq
\beta_{20}(\omega)+\sum_{i=1}^2\beta_{2i}(t,\omega).
\end{align*}
\end{lemma}
\begin{proof}
For proof see Lemma 7, section 4 of \cite{SS4}.
\end{proof}
By taking the Helmhotz-Hodge projection $P_{\mathbf{H}}$, we get
$$B(\u,\v)=P_{\mathbf{H}}(\u\cdot\nabla\v),
B_1(\w)\v=P_{\mathbf{H}}(\w\cdot\nabla\v),
B_2(\w)\v=P_{\mathbf{H}}(\v\cdot\nabla\w), \;\forall\; \u,\v\in
\mathbf{V}.$$

\begin{remark}\label{rem121}
For the notational convenience from now onwards we use the symbols
$\overline{\beta}_{10}(t)=\mathbb{E}[\beta_{10}(t,\omega)],$
$\overline{\beta}_{1j}=\mathbb{E}[\beta_{1j}(\omega)]$ for $j=1,2$,
$\overline{\beta}^2_{20}=\mathbb{E}[\beta^2_{20}(\omega)]$ and
$\overline{\beta}_{2j}(t)=\mathbb{E}[\beta_{2j}(t,\omega)]$ for
$j=1,2$.
\end{remark}

\section{Construction of the Perturbed Vector Field}\setcounter{equation}{0}
Let us now consider the system (\ref{1})-(\ref{5}) and introduce the
following change of variables:
$$\mathbf{u}(\mathbf{x},t,\omega)=\mathbf{v}(\mathbf{x},t,\omega)+\mathbf{w}(\mathbf{x},t,\omega)\textrm{ and }
p(\x,t,\omega)=q(\x,t,\omega)+P(\x,t,\omega),$$ where
$\mathbf{w}(\cdot,\cdot,\cdot)$ is the basic vector field. The
regularity properties of this basic vector field imply the following
properties on $\beta_{1j}$ and $\beta_{2j},$ for $j=0,1,2$:
$$\beta_{11}(\omega),\beta_{12}(\omega),\beta_{20}(\omega)\in\mathrm{L}^2(\Omega)\textrm{
and }
\beta_{10}(t,\omega),\beta_{21}(t,\omega),\beta_{22}(t,\omega)\in\mathrm{L}^2(\Omega;\mathrm{L}^2(0,T)).$$
Our problem is to find
$$(\mathbf{v},q):\Theta\times
[0,T]\times\Omega\rightarrow\mathbb{R}^2\times \mathbb{R}$$ such
that
\begin{align}\label{6}
\v_t+\v\cdot\nabla \v+\v\cdot\nabla\w +\w\cdot\nabla \v &=-\nabla q
+\nu \Delta \v +\mathbf{f}_{\w}+\dot{\mathscr{G}}(\x,t)\nonumber\\
&\quad\textrm{ in }\Theta\times (0,T)\times\Omega,\\ \nabla\cdot
\v&= 0\textrm{ in } \Theta \times (0,T)\times\Omega,\\
\v\big|_{\partial \Theta\times[0,T]\times\Omega}&=
0,\\\v&\rightarrow 0\textrm{ as
} |\mathbf{x}|\rightarrow \infty\textrm{ in }\mathbf{O}_i, i=1,2,\\
\v(\x,0,\omega)&=0,(\x,\omega)\in \Theta\times\Omega\quad \textrm{ and }\\
\int_{\Gamma}\v\cdot\n\d S&=0,
\end{align}
where
\begin{align}\label{7}
\mathbf{f}_{\w}=\nu\Delta\w-\w\cdot\nabla\w-\w_{t}-\nabla P.
\end{align}
From the construction of the basic vector field it is clear that the
supp$\{\mathbf{f}_{\w}\}\Subset\Theta$. Also one can notice that
$\mathbf{f}_{\w}\equiv 0$ for $\mathbf{O}_i\backslash
\hat{\Theta}_0$. Moreover, form the estimate
$$\w_t-\nu\Delta\w+\w\cdot\nabla\w\in
\mathrm{L}^2\left(\Omega;\mathrm{L}^2(0,T;H^{-1}(\Theta_0))\right),$$
we have $\mathbf{f}_{\w}\in
\mathrm{L}^2\left(\Omega;\mathrm{L}^2(0,T;\mathbf{V}')\right).$
A-priori estimates for the solution can be obtained by assuming
smoothness of $\v$ and sufficient decay at infinity.

Let us now take the Helmhotz-Hodge projection on the equation
(\ref{6}) and using the fact that $(\nabla q,\v)=0$, $\forall\;
\v\in \mathbf{H}$, one can reduce the equation (\ref{6}) to
\begin{align}\label{eqn111}
\d\v&=(F(\v(t))+P_{\mathbf{H}}\mathbf{f}_{\w}(t))\d t+g(t)\d W(t),\\
\textrm{where }\;\;F(\v(t))&=-\nu
A\v-B(\v,\v)-B_1(\w)\v-B_2(\w)\v\nonumber\\ A\v&=-P_{\mathbf{H}}\Delta\v\nonumber\\
B(\u,\v)&=P_{\mathbf{H}}(\u\cdot\nabla\v)\nonumber\\
B_1(\w)\v&=P_{\mathbf{H}}(\w\cdot\nabla\v)\nonumber\\
B_2(\w)\v&=P_{\mathbf{H}}(\v\cdot\nabla\w).\nonumber
\end{align}
See Appendix of Mikulevicius and Rozovskii \cite{MR} for more
details about the properties of the projection operator.

Let the noise process be represented as a series $\d
\mathscr{G}_k=\sum_k g_k(\x,t)\d W_k(t,\omega)$, where
$g=(g_1,g_2,\cdots)$ is and $\ell^2$- valued function and $W_k$ are
mutually independent standard one dimensional Brownian motions. The
stochastic term $g\d W$ is thus an $\mathbf{H}$ - valued Wiener
process with a trace-class covariance operator denoted by
$g^*g=g^*g(t)$ given by
\begin{align*}
\left(g^*g(t)u,v\right)&=\sum_k
\left(g_k(t),u\right)\left(g_k(t),v\right),\nonumber\\
\textrm{Tr}(g^*g(t))&=\sum_k \left|g_k(t)\right|^2<\infty.
\end{align*}
This means that the mapping
\begin{equation*}
v\rightarrow \int_0^T\left(g(t)\d W(t),v\right):=\sum_k\int_0^T
\left(g_k(t),v\right)\d W_k(t)
\end{equation*}
is a continuous linear functional on $\mathbf{H}$ with probability
$1$ and the noise is the formal time-derivative of the process
$P_{\mathbf{H}}\mathscr{G}(t)=\int_0^t g(t)\d W(t)$. A
multiplicative noise of the form $g(\u(\x,t))\d W(t)$, where
$g(\u(\x,t))$ is a continuous operator from $\mathbf{V}$ into
$\mathrm{L}^2(0,T;\ell^2(\mathbf{H}))$, can also be considered as
the random forcing. Here $\dot{\mathscr{G}}(\cdot)$ is formally
written and it is the time-derivative of $\int_0^{\cdot}g(s)\d
W(s)$. We have $\int_0^{\cdot}g(s)\d
W(s)\in\mathrm{L}^2(\Omega;C([0,T];\mathbf{H}))$ implies its
time-derivative $\dot{\mathscr{G}}(\cdot)$ satisfies
$\dot{\mathscr{G}}(\cdot)\in\mathrm{L}^2(\Omega;\mathrm{W}^{-1,\infty}(0,T;\mathbf{H})),$
since $\partial_t$ is linear continuous from $C([0,T];\mathbf{H})$
into $W^{-1,\infty}(0,T;\mathbf{H})$.

\section{A-Priori Estimates for the Perturbed Vector
Field}\setcounter{equation}{0} Let
$\mathbf{H}_n=span\{e_1,e_2,\cdots,e_n\},$ where $\{e_j\}$ is any
fixed orthonormal basis in $\mathbf{H}$ with $e_j\in D(A)$. Let
$P_n$ denote the orthogonal projection of $\mathbf{H}$ into
$\mathbf{H}_n$. Define $\v^n=P_n\v$. Let us define $\v^n$ as the
solution of the following stochastic differential equation in the
variational form such that for each $v\in\mathbf{H}_n$,
\begin{align}\label{main}
\d(\v^n(t),v)&=(F(\v^n(t)),v)\d
t+(P_{\mathbf{H}}\mathbf{f}_{\w}(t),v)\d
t+\sum_k(g_k(t),v)\d W_k(t),\\
\v^n(0)&=P_n\v(0)\;=\;0,\nonumber
\end{align}
where $F(\v)=-\nu A\v-B(\v,\v)-B_1(\w)\v-B_2(\w)\v.$

It can be shown that for all $n\geq 1$, there exists an adapted
process $\v^n\in C([0,T];\H_n)$ a. s. such that $\v^n$ satisfies
(\ref{main}). For proof see corollary 2.1 of Albeverio,
Brze\'{z}niak, Wu \cite{ABW}, page 128 of Brze\'{z}niak, Hausenblas,
Zhu \cite{BHZ}, Lemma 3.1 of Capi\'{n}ski and Peszat \cite{CaP},
Proposition 3.2 of Manna, Menaldi and Sritharan \cite{MMS}, Theorem
3.1.1 of Pr\'{e}v\^{o}t and R\"{o}ckner \cite{PrR}.

\begin{theorem}\label{energy1}
Let $\Theta$ be an admissible channel domain. Under the above
mathematical setting, let $\v^n(t)$ be an adapted process in
$C([0,T];\H_n)$ which solves the stochastic ODE (\ref{main}). Then
we have the following a-priori estimates: For $\delta>0$
\begin{align}\label{35}
&\mathbb{E}\left|\v^n(\cdot,t)\right|^2e^{-\delta t}+\nu
\mathbb{E}\left(\int_0^t\|\v^n(\cdot,s)\|^2e^{-\delta s}\d s\right)
\nonumber\\&\qquad\quad \leq
C_1\left(T,\nu,\|\mathscr{F}\|_{_{\mathrm{L}^2(\Omega;H^1(0,T))}},\delta,\int_0^T\Tr(g^*g)(t)e^{-\delta
t}\d t\right)
\end{align}
and
\begin{align}\label{36}
&\mathbb{E}\left[\sup_{0\leq t\leq
T}\left|\v^n(\cdot,t)\right|^2e^{-\delta t}\right]+2\nu
\mathbb{E}\left(\int_0^T\|\v^n(\cdot,t)\|^2e^{-\delta t}\d s\right)
\nonumber\\&\quad\quad\leq
C_2\left(T,\nu,\|\mathscr{F}\|_{_{\mathrm{L}^2(\Omega;H^1(0,T))}},\delta,\int_0^T\Tr(g^*g)(t)e^{-\delta
t}\d t\right).
\end{align}
\end{theorem}

\begin{proof}
Let us consider the function $e^{-\delta t}|x|^2$ and apply
It\^{o}'s lemma (see Theorem 1, page 155 Gy\"{o}ngy and Krylov
\cite{GyKr}, Metivier \cite{Me}) to the process $\v^n(t)$, for any
$\delta>0$
\begin{align}\label{37}
&\d\left[e^{-\delta t}|\v^n|^2\right]=-\delta e^{-\delta
t}|\v^n|^2\d
t-2\left[b(\v^n,\v^n,\v^n)+b(\v^n,\w,\v^n)\right.\nonumber\\&\qquad\qquad\qquad\left.+b(\w,\v^n,\v^n)+\nu(\Delta
\v^n,\v^n)\right]e^{-\delta t}\d
t+\left[(P_{\mathbf{H}}\mathbf{f}_{\w},\v^n)\right]e^{-\delta t}\d
t\nonumber\\&\qquad\qquad\qquad+2\sum_k(g_k(t),\v^n)e^{-\delta t}\d
W_k(t)+e^{-\delta t}\mathrm{Tr}(g^*g)(t)\d t.
\end{align}
By applying the properties of the trilinear form,
$(\Delta\v^n,\v^n)=-\|\v^n\|^2$ and
$2|(P_{\mathbf{H}}\mathbf{f}_{\w},\v^n)|\leq\delta
|\v^n|^2+\frac{1}{\delta}|\mathbf{f}_{\w}|^2$ in (\ref{37}), we get
\begin{align}\label{39}
&\d\left[[e^{-\delta t}|\v^n|^2\right]+2\nu\|\v^n\|^2e^{-\delta t}\d
t\leq -2b(\v^n,\w,\v^n)e^{-\delta t}\d
t+\frac{1}{\delta}|\mathbf{f}_{\w}|^2e^{-\delta t}\d
t\nonumber\\&\qquad\qquad\qquad\qquad+e^{-\delta
t}\mathrm{Tr}(g^*g)(t)\d t+2\sum_k(g_k(t),\v^n)e^{-\delta t}\d
W_k(t).
\end{align}
Let $\mathcal{G}$ be the $\sigma$-algebra generated by $\v^n$ in the
probability space $(\Omega;\mathcal{F},\mathcal{F}_t,\mathbb{P})$.
Since we have assumed that the external random forcing and the
random flux are mutually independent processes,  the constructed
vector field $\w$ is independent of $\mathcal{G}$.

Let us define the stopping time $\tau_N$ by
$$\tau_N=\inf\left\{t:|\v^n(t)|^2+\int_0^t\|\v^n(s)\|^2\d
s>N\right\}.$$  Note that $\tau_N$ is adapted to $\mathcal{G}$ and
$\w$ is independent of $\mathcal{G}$.

Now let us integrate the inequality (\ref{39}) in $t$ from $0$ to
$\t$ to get
\begin{align}\label{40a}
&|\v^n(\t)|^2e^{-\delta t}+2\nu\int_0^{\t}\|\v^n(s)\|^2e^{-\delta
s}\d s\nonumber\\&\leq
2\int_0^{\t}\left|\left(B_2(\w)\v^n,\v^n\right)(s)\right|e^{-\delta
s}\d s+\frac{1}{\delta}\int_0^{t}|\mathbf{f}_{\w}|^2e^{-\delta s}\d
s\nonumber\\&\quad+\int_0^{t}\mathrm{Tr}(g^*g)(s)e^{-\delta s}\d
s+2\sum_k\int_0^{\t}(g_k(s),\v^n)e^{-\delta s}\d W_k(s).
\end{align}
Also for the term $\left(B_2(\w)\v^n,\v^n\right)=b(\v^n,\w,\v^n)$,
we have the estimate
$$|b(\u,\w,\v)|\leq
\beta_{20}(\omega)|\u|^{1/2}\|\u\|^{1/2}|\v|^{1/2}\|\v\|^{1/2}+\left(\beta_{21}(t,\omega)+\beta_{22}(t,\omega)\right)|\u||\v|.$$
By using the above estimate and using Young's inequality, we have
for all $\v^n\in\mathbf{V}$
\begin{align*}
|b(\v^n,\w,\v^n)|&\leq
\beta_{20}(\omega)|\v^n|\|\v^n\|+\left(\beta_{21}(t,\omega)+\beta_{22}(t,\omega)\right)|\v^n|^2\nonumber\\
&\leq
\frac{\nu}{2}\|\v^n\|^2+\left(\frac{\beta^2_{20}(\omega)}{2\nu}+\beta_{21}(t,\omega)+\beta_{22}(t,\omega)\right)|\v^n|^2.
\end{align*}
Hence (\ref{40a}) becomes
\begin{align}\label{40b}
&|\v^n(\t)|^2e^{-\delta t}+\nu\int_0^{\t}\|\v^n(s)\|^2e^{-\delta
s}\d s\nonumber\\&\leq
2\int_0^{\t}\left(\frac{\beta^2_{20}(\omega)}{2\nu}+\beta_{21}(s,\omega)+\beta_{22}(s,\omega)\right)|\v^n|^2e^{-\delta
s}\d s+\frac{1}{\delta}\int_0^{t}|\mathbf{f}_{\w}|^2e^{-\delta s}\d
s\nonumber\\&\quad+\int_0^{t}\mathrm{Tr}(g^*g)(s)e^{-\delta s}\d
s+2\sum_k\int_0^{\t}(g_k(s),\v^n)e^{-\delta s}\d W_k(s).
\end{align}
Let us take conditional expectation on both sides of (\ref{40b})
with respect to the $\sigma$-algebra $\mathcal{G}$ to get
\begin{align}\label{40c}
&\mathbb{E}\left(\left.|\v^n(\t)|^2e^{-\delta
t}\right|\mathcal{G}\right)+\nu\mathbb{E}\left(\left.\int_0^{\t}\|\v^n(s)\|^2e^{-\delta
s}\d s\right|\mathcal{G}\right)\nonumber\\&\leq
2\mathbb{E}\left(\left.\int_0^{\t}\left(\frac{\beta^2_{20}(\omega)}{2\nu}+\beta_{21}(s,\omega)+\beta_{22}(s,\omega)\right)|\v^n|^2e^{-\delta
s}\d
s\right|\mathcal{G}\right)\nonumber\\&\quad+\frac{1}{\delta}\mathbb{E}\left(\left.\int_0^{t}|\mathbf{f}_{\w}|^2e^{-\delta
s}\d
s\right|\mathcal{G}\right)+\int_0^{t}\mathrm{Tr}(g^*g)(s)e^{-\delta
s}\d
s\nonumber\\&\quad+2\mathbb{E}\left(\left.\sum_k\int_0^{\t}(g_k(s),\v^n)e^{-\delta
s}\d W_k(s)\right|\mathcal{G}\right).
\end{align}
In (\ref{40c}), let us use the $\mathcal{G}$-measurability of $\v^n$
and independence of $\w$ with respect to $\mathcal{G}$ to get
\begin{align}\label{40d}
&|\v^n(\t)|^2e^{-\delta t}+\nu\int_0^{\t}\|\v^n(s)\|^2e^{-\delta
s}\d s\nonumber\\&\leq
2\mathbb{E}\left(\left.\int_0^{\t}\left(\frac{\beta^2_{20}(\omega)}{2\nu}+\beta_{21}(s,\omega)+\beta_{22}(s,\omega)\right)|\v^n|^2e^{-\delta
s}\d
s\right|\mathcal{G}\right)\nonumber\\&\quad+\frac{1}{\delta}\mathbb{E}\left(\int_0^{t}|\mathbf{f}_{\w}|^2e^{-\delta
s}\d s\right)+\int_0^{t}\mathrm{Tr}(g^*g)(s)e^{-\delta s}\d
s\nonumber\\&\quad+2\sum_k\int_0^{\t}(g_k(s),\v^n)e^{-\delta s}\d
W_k(s).
\end{align}
Let us take expectation on both sides of the inequality (\ref{40d})
to obtain
\begin{align}\label{40e}
&\mathbb{E}\left(|\v^n(\t)|^2e^{-\delta
t}\right)+\nu\mathbb{E}\left(\int_0^{\t}\|\v^n(s)\|^2e^{-\delta s}\d
s\right)\nonumber\\&\leq
2\mathbb{E}\left\{\mathbb{E}\left(\left.\int_0^{\t}\left(\frac{\beta^2_{20}(\omega)}{2\nu}+\beta_{21}(s,\omega)+\beta_{22}(s,\omega)\right)|\v^n|^2e^{-\delta
s}\d
s\right|\mathcal{G}\right)\right\}\nonumber\\&\quad+\frac{1}{\delta}\mathbb{E}\left(\int_0^{t}|\mathbf{f}_{\w}|^2e^{-\delta
s}\d s\right)+\int_0^{t}\mathrm{Tr}(g^*g)(s)e^{-\delta s}\d
s\nonumber\\&\quad+2\mathbb{E}\left(\sum_k\int_0^{\t}(g_k(s),\v^n)e^{-\delta
s}\d W_k(s)\right).
\end{align}
Note that the last term in the inequality (\ref{40e}) is a
martingale with zero expectation, one obtains
\begin{align}\label{40}
&\mathbb{E}\left[|\v^n(\t)|^2e^{-\delta
t}\right]+\nu\mathbb{E}\left(\int_0^{\t}\|\v^n(s)\|^2e^{-\delta s}\d
s\right)\nonumber\\&\leq
2\mathbb{E}\left\{\mathbb{E}\left(\left.\int_0^{\t}\left(\frac{\beta^2_{20}(\omega)}{2\nu}+\beta_{21}(s,\omega)+\beta_{22}(s,\omega)\right)|\v^n|^2e^{-\delta
s}\d
s\right|\mathcal{G}\right)\right\}\nonumber\\&\quad+\frac{1}{\delta}\mathbb{E}\left(\int_0^{t}|\mathbf{f}_{\w}|^2e^{-\delta
s}\d s\right)+\int_0^{t}\mathrm{Tr}(g^*g)(s)e^{-\delta s}\d s.
\end{align}
Now let us consider
$$\mathbb{E}\left\{\mathbb{E}\left.\left(\int_0^{\t}\left[\frac{\beta^2_{20}(\omega)}{2\nu}+\beta_{21}(s,\omega)+\beta_{22}(s,\omega)\right]|\v^n|^2e^{-\delta
s}\d s\right|{\mathcal{G}}\right)\right\}$$ and use the properties
of conditional expectation, independence of  $\w$ with the
$\sigma$-algebra $\mathcal{G}$, generated by $\v^n$ and H\"{o}lder's
inequality to get
\begin{align}\label{qe1}
&\mathbb{E}\left\{\mathbb{E}\left.\left(\int_0^{\t}\left[\frac{\beta^2_{20}(\omega)}{2\nu}+\beta_{21}(s,\omega)+\beta_{22}(s,\omega)\right]|\v^n|^2e^{-\delta
s}\d s\right|{\mathcal{G}}\right)\right\}\nonumber\\& =
\mathbb{E}\left.\left\{\mathbb{E}\left(\int_0^{T}\left(\chi_{[0,\t]}|\v^n(s)|^2e^{-\delta
s}\right)
\left[\frac{\beta^2_{20}(\omega)}{2\nu}+\beta_{21}(s,\omega)+\beta_{22}(s,\omega)\right]\d
s\right|{\mathcal{G}}\right)\right\}\nonumber\\& =
\mathbb{E}\left\{\int_0^{T}\left(\chi_{[0,\t]}|\v^n(s)|^2e^{-\delta
s}\right)
\mathbb{E}\left(\left.\left[\frac{\beta^2_{20}(\omega)}{2\nu}+\beta_{21}(s,\omega)+\beta_{22}(s,\omega)\right]\right|{\mathcal{G}}\right)\d
s\right\}\nonumber\\&
=\mathbb{E}\left\{\int_0^{T}\left(\chi_{[0,\t]}|\v^n(s)|^2e^{-\delta
s}\right)
\mathbb{E}\left(\left[\frac{\beta^2_{20}(\omega)}{2\nu}+\beta_{21}(s,\omega)+\beta_{22}(s,\omega)\right]\right)\d
s\right\}\nonumber\\&=
\mathbb{E}\left(\int_0^{T}\left(\chi_{[0,\t]}|\v^n(s)|^2e^{-\delta
s}\right)
\left[\frac{\overline{\beta}^2_{20}}{2\nu}+\overline{\beta}_{21}(s)+\overline{\beta}_{22}(s)\right]\d
s\right)\nonumber\\& \leq \sup_{0\leq t\leq
T}\left[\frac{\overline{\beta}^2_{20}}{2\nu}+\overline{\beta}_{21}(t)+\overline{\beta}_{22}(t)\right]
\mathbb{E}\left(\int_0^{T}\left(\chi_{[0,\t]}|\v^n(s)|^2e^{-\delta
s}\right) \d s\right)\nonumber\\& = \sup_{0\leq t\leq
T}\left[\frac{\overline{\beta}^2_{20}}{2\nu}+\overline{\beta}_{21}(t)+\overline{\beta}_{22}(t)\right]
\mathbb{E}\left(\int_0^{\t}\left(|\v^n(s)|^2e^{-\delta s}\right) \d
s\right).
\end{align}
Note here that $\sup_{0\leq t\leq
T}\beta_{2j}^2(t,\omega)=\eta_{2j}^2(\omega),$ for $j=1,2$ is
bounded from (\ref{qeta}) and
$\eta_{2j}(\cdot)\in\mathrm{L}^2(\Omega)$. Hence, from inequality
(\ref{40}), we get
\begin{align}\label{41}
&\mathbb{E}\left[|\v^n(\t)|^2e^{-\delta
t}\right]+\nu\mathbb{E}\left(\int_0^{\t}\|\v^n(s)\|^2e^{-\delta s}\d
s\right) \nonumber\\&\leq
\frac{1}{\delta}\mathbb{E}\left(\int_0^{t}|\mathbf{f}_{\w}|^2e^{-\delta
s}\d s\right)+\int_0^{t}\mathrm{Tr}(g^*g)(s)e^{-\delta s}\d
s\nonumber\\&\quad +\sup_{0\leq s\leq
T}\left[\frac{\overline{\beta}^2_{20}}{\nu}+2\overline{\beta}_{21}(s)+2\overline{\beta}_{22}(s)\right]
\mathbb{E}\left(\int_0^{\t}\left(|\v^n(s)|^2e^{-\delta s}\right) \d
s\right).
\end{align}
In particular, we have
\begin{align*}
&\mathbb{E}\left[|\v^n(\t)|^2e^{-\delta t}\right]\leq
\frac{1}{\delta}\mathbb{E}\left(\int_0^{t}|\mathbf{f}_{\w}|^2e^{-\delta
s}\d s\right)+\int_0^{t}\mathrm{Tr}(g^*g)(s)e^{-\delta s}\d
s\\&\qquad\quad\quad+\sup_{0\leq s\leq
T}\left[\frac{\overline{\beta}^2_{20}}{\nu}+2\overline{\beta}_{21}(s)+2\overline{\beta}_{22}(s)\right]
\mathbb{E}\left(\int_0^{\t}\left(|\v^n(s)|^2e^{-\delta s}\right) \d
s\right).
\end{align*}
An application of Gronwall's lemma yields
\begin{align}\label{43}
\mathbb{E}\left[|\v^n(\t)|^2e^{-\delta t}\right]&\leq
\left(\frac{1}{\delta}\mathbb{E}\left(\int_0^{t}|\mathbf{f}_{\w}|^2e^{-\delta
s}\d s\right)+\int_0^{t}\mathrm{Tr}(g^*g)(s)e^{-\delta s}\d
s\right)\nonumber\\&\quad\exp\left(\sup_{0\leq t\leq
T}\left[\frac{\overline{\beta}^2_{20}}{\nu}+2\overline{\beta}_{21}(t)+2\overline{\beta}_{22}(t)\right]T\right).
\end{align}
By applying (\ref{43}) in (\ref{41}), we get
\begin{align}\label{44}
&\mathbb{E}\left[|\v^n(\t)|^2e^{-\delta
t}\right]+\nu\mathbb{E}\left(\int_0^{\t}\|\v^n(s)\|^2e^{-\delta s}\d
s\right)\nonumber\\&\leq
\left(\frac{1}{\delta}\mathbb{E}\left(\int_0^{t}|\mathbf{f}_{\w}|^2e^{-\delta
s}\d s\right)+\int_0^{t}\mathrm{Tr}(g^*g)(s)e^{-\delta s}\d
s\right)\nonumber\\&\quad\left[1+\exp\left(\sup_{0\leq t\leq
T}\left[\frac{\overline{\beta}^2_{20}}{\nu}+2\overline{\beta}_{21}(t)+2\overline{\beta}_{22}(t)\right]T
\right)\right.\nonumber\\&\quad\left.\left(\sup_{0\leq t\leq
T}\left[\frac{\overline{\beta}^2_{20}}{\nu}+2\overline{\beta}_{21}(t)+2\overline{\beta}_{22}(t)\right]T\right)\right].
\end{align}
Now let us take $N\rightarrow\infty$ so that $\t\rightarrow t$ and
by using the estimates (\ref{eqn1004}) and (\ref{eqn1005}) to obtain
\begin{align*}
&\mathbb{E}\left[|\v^n(t)|^2e^{-\delta
t}\right]+\nu\mathbb{E}\left(\int_0^t\|\v^n(s)\|^2e^{-\delta s}\d
s\right)\\&\qquad\qquad\qquad\leq
C_1\left(T,\nu,\|\mathscr{F}\|_{_{\mathrm{L}^2(\Omega;H^1(0,T))}},\delta,\int_0^{t}\mathrm{Tr}(g^*g)(s)e^{-\delta
s}\d s\right).
\end{align*}
In the estimate (\ref{40d}), taking supremum up to $\T$ before
taking the expectation, we obtain
\begin{align}\label{46}
&\mathbb{E}\left[\sup_{0\leq t\leq \T}|\v^n(t)|^2e^{-\delta
t}\right]+\nu \mathbb{E}\left(\int_0^{\T}\|\v^n(t)\|^2e^{-\delta
t}\d t\right)\nonumber\\&\leq
\frac{1}{\delta}\mathbb{E}\left(\int_0^{T}|\mathbf{f}_{\w}|^2e^{-\delta
t}\d t\right)+\int_0^{T}\Tr(g^*g)(t)e^{-\delta t}\d
t\nonumber\\&\quad+2\mathbb{E}\left[\sup_{0\leq t\leq
\T}\int_0^{t}\sum_k(g_k(s),\v^n(s))e^{-\delta s}\d
W_k(s)\right]\nonumber\\&\quad +2\mathbb{E}\left\{\sup_{0\leq t\leq
\T}
\mathbb{E}\left(\left.\int_0^{t}\left(\frac{\beta^2_{20}(\omega)}{2\nu}+\beta_{21}(s,\omega)+\beta_{22}(s,\omega)\right)|\v^n|^2e^{-\delta
s}\d s\right|\mathcal{G}\right)\right\}.
\end{align}
Now by applying Burkholder-Davis-Gundy inequality and Young's
inequality to the term $2\mathbb{E}\left[\sup_{0\leq t\leq
\T}\int_0^{t}\sum_k(g_k(s),\v^n(s))e^{-\delta s}\d W_k(s)\right],$
one gets
\begin{align}\label{47}
&2\mathbb{E}\left[\sup_{0\leq t\leq
\T}\left|\int_0^t\sum_k(g_k(s),\v^n(s))e^{-\delta s}\d
W_k(s)\right|\right]\nonumber\\&\leq
2\sqrt{2}\mathbb{E}\left(\int_0^{\T}\Tr(g^*g)(t)|\v^n(t)|^2e^{-2\delta
t}\d t\right)^{1/2}\nonumber\\&\leq 2\sqrt{2}
\mathbb{E}\left[\left(\sup_{0\leq t\leq \T}|\v^n(t)|^2e^{-\delta
t}\right)^{1/2}\left(\int_0^{\T}\Tr(g^*g)(t)e^{-\delta t}\d
t\right)^{1/2}\right]\nonumber\\&\leq
\frac{1}{2}\mathbb{E}\left(\sup_{0\leq t\leq
\T}|\v^n(t)|^2e^{-\delta t}\right)+4\int_0^{T}\Tr(g^*g)(t)e^{-\delta
t}\d t.
\end{align}
From (\ref{qe1}), it can be easily seen that
\begin{align}\label{qe2}
&\mathbb{E}\left\{\left.\sup_{0\leq t\leq
\T}\mathbb{E}\left(\int_0^{t}\left[\frac{\beta^2_{20}(\omega)}{2\nu}+\beta_{21}(s,\omega)+\beta_{22}(s,\omega)\right]|\v^n|^2e^{-\delta
s}\d s\right|{\mathcal{G}}\right)\right\}\nonumber\\&\leq
\sup_{0\leq t\leq
T}\left[\frac{\overline{\beta}^2_{20}}{2\nu}+\overline{\beta}_{21}(t)+\overline{\beta}_{22}(t)\right]
\mathbb{E}\left(\int_0^{\T}\sup_{0\leq s\leq t}|\v^n(s)|^2e^{-\delta
s} \d s\right).
\end{align} Hence (\ref{46})
becomes
\begin{align}\label{48}
&\mathbb{E}\left[\sup_{0\leq t\leq \T}|\v^n(t)|^2e^{-\delta
t}\right]+2\nu\mathbb{E}\left(\int_0^T\|\v^n(t)\|^2e^{-\delta t}\d
t\right)\nonumber\\&\leq
\frac{2}{\delta}\mathbb{E}\left(\int_0^T|\mathbf{f}_{\w}|^2e^{-\delta
t}\d t\right)+10\int_0^{T}\Tr(g^*g)(t)e^{-\delta t}\d
t\nonumber\\&\quad+2\sup_{0\leq t\leq
T}\left[\frac{\overline{\beta}^2_{20}}{2\nu}+\overline{\beta}_{21}(t)+\overline{\beta}_{22}(t)\right]
\mathbb{E}\left(\int_0^{\T}\sup_{0\leq s\leq t}|\v^n(s)|^2e^{-\delta
s} \d s\right).
\end{align}
In particular, we have
\begin{align*}
&\mathbb{E}\left[\sup_{0\leq t\leq \T}|\v^n(t)|^2e^{-\delta
t}\right]\nonumber\\&\leq
\frac{2}{\delta}\mathbb{E}\left(\int_0^{T}|\mathbf{f}_{\w}|^2e^{-\delta
t}\d t\right)+10\int_0^{T}\Tr(g^*g)(t)e^{-\delta t}\d
t\nonumber\\&\quad+\sup_{0\leq t\leq
T}\left[\frac{\overline{\beta}^2_{20}}{\nu}+2\overline{\beta}_{21}(t)+2\overline{\beta}_{22}(t)\right]
\mathbb{E}\left(\int_0^{\T}\left(\sup_{0\leq s\leq
t}|\v^n(s)|^2e^{-\delta s} \right)\d s\right).
\end{align*}
Applying Gronwall's inequality, one gets
\begin{align}\label{50}
&\mathbb{E}\left[\sup_{0\leq t\leq \T}|\v^n(t)|^2e^{-\delta
t}\right]\nonumber\\&\leq
\left(\frac{2}{\delta}\mathbb{E}\left(\int_0^{T}|\mathbf{f}_{\w}|^2e^{-\delta
t}\d t\right)+10\int_0^{T}\Tr(g^*g)(t)e^{-\delta t}\d
t\right)\nonumber\\&\quad\exp\left(\sup_{0\leq t\leq
T}\left[\frac{\overline{\beta}^2_{20}}{\nu}+2\overline{\beta}_{21}(t)+2\overline{\beta}_{22}(t)\right]T\right).
\end{align}
Hence by applying (\ref{50}) in (\ref{48}), we deduce that
\begin{align}\label{51}
&\mathbb{E}\left[\sup_{0\leq t\leq \T}|\v^n(t)|^2e^{-\delta
t}\right]+2\nu\int_0^{\T}\mathbb{E}\|\v^n(t)\|^2e^{-\delta t}\d
t\nonumber\\&\leq
\left(\frac{2}{\delta}\mathbb{E}\left(\int_0^{T}|\mathbf{f}_{\w}|^2e^{-\delta
t}\d t\right)+10\int_0^{T}\Tr(g^*g)(t)e^{-\delta t}\d
t\right)\nonumber\\&\quad\left[1+2\exp\left(\sup_{0\leq t\leq
T}\left[\frac{\overline{\beta}^2_{20}}{\nu}+2\overline{\beta}_{21}(t)+2\overline{\beta}_{22}(t)\right]T
\right)\right.\nonumber\\&\quad\left.\left(\sup_{0\leq t\leq
T}\left[\frac{\overline{\beta}^2_{20}}{\nu}+2\overline{\beta}_{21}(t)+2\overline{\beta}_{22}(t)\right]T\right)\right].
\end{align}
Let us now take $N\rightarrow\infty$ so that $\T\rightarrow T$ and
by using (\ref{eqn1004}) and (\ref{eqn1005}) to get
\begin{align*}
&\mathbb{E}\left[\sup_{0\leq t\leq T}|\v|^2e^{-\delta
t}\right]+2\nu\mathbb{E}\left(\int_0^T\|\v\|^2e^{-\delta t}\d
t\right) \nonumber\\&\leq
C_2\left(\delta,T,\nu,\|\mathscr{F}\|_{_{\mathrm{L}^2(\Omega;H^1(0,T))}},\int_0^{T}\Tr(g^*g)(t)e^{-\delta
t}\d t\right).
\end{align*}
Hence we get the required result.
\end{proof}

\begin{theorem}\label{energy3}
Let $\Theta$ be an admissible channel domain. Under the above
mathematical setting, let $\v^n(t)$ be an adapted process in
$C([0,T];\H_n)$ which solves the stochastic ODE (\ref{main}). Then
we have the following a-priori estimates:
\begin{align}\label{8}
&\mathbb{E}\left|\v^n(\cdot,t)\right|^2+\nu
\mathbb{E}\left(\int_0^t\|\v^n(\cdot,s)\|^2\d
s\right)\nonumber\\&\qquad\qquad \leq
C_3\left(T,\nu,\|\mathscr{F}\|_{_{\mathrm{L}^2(\Omega;H^1(0,T))}},\int_0^T
\mathrm{Tr}(g^*g)(t)\d t\right)
\end{align}
and
\begin{align}\label{9}
&\mathbb{E}\left(\sup_{0\leq t\leq
T}\left|\v^n(\cdot,t)\right|^2+2\nu \int_0^T\|\v^n(\cdot,t)\|^2\d
s\right)\nonumber\\&\qquad\qquad\leq
C_4\left(T,\nu,\|\mathscr{F}\|_{_{\mathrm{L}^2(\Omega;H^1(0,T))}},\int_0^T
\mathrm{Tr}(g^*g)(t)\d t\right).
\end{align}
\end{theorem}

\begin{proof}
Apply It\^{o}'s lemma (see Theorem 1, page 155 Gy\"{o}ngy and Krylov
\cite{GyKr}, Metivier \cite{Me})  to the function $|x|^2$ and to the
process $\v^n(t)$. Proceeding similarly as in Theorem \ref{energy1}
and using the inequality
$$2|(P_{\mathbf{H}}\mathbf{f}_{\w},\v^n)|\leq
\|\v^n\|^2+\|\mathbf{f}_{\w}\|_{\V'}^2,$$ we get the required
result.
\end{proof}

\section{Existence and Uniqueness of Strong Solutions for the Perturbed Vector Field}\setcounter{equation}{0}
Monotonicity arguments were first used by Krylov and Rozovskii
\cite{KR} to prove the existence and uniqueness of the strong
solutions for a wide class of stochastic evolution equations (under
certain assumptions on the drift and diffusion coefficients), which
in fact is the refinement of the previous results by Pardoux \cite{
Pa1, Pa2} (also see  Metivier \cite{Me}) and also the generalization
of the results by Bensoussan and Temam \cite{Be}. Menaldi and
Sritharan \cite{MS} further developed this theory for the case when
the sum of the linear and nonlinear operators are locally monotone.

In this section, we will prove the local monotonicity of the sum of
the Stokes operator and the inertia term of the perturbed vector
field, and following Menaldi and Sritharan \cite{MS} we use a
generalization of Minty-Browder technique to prove the existence and
uniqueness result that avoids compactness method and hence
applicable for unbounded domains directly.

We will start this section with a lemma known as the
Gagliardo-Nirenberg inequality (section 1.2, Theorem 2.1 in
DiBenedetto \cite{ED}).

\begin{lemma}\label{GN}
Let $\varphi\in W^{1,p}(\mathbb{R}^n),$ where $n$ is the dimension
of the space. For every fixed number $p,s\geq 1,$ there exists a
constant $C$ depending only upon $n,p$ and $s$ such that
$$\|\varphi\|_{\mathrm{L}^q(\mathbb{R}^n)}\leq C\|\nabla
\varphi\|^{\alpha}_{\mathrm{L}^p(\mathbb{R}^n)}\|\varphi\|_{\mathrm{L}^s(\mathbb{R}^n)}^{1-\alpha},$$
where $\alpha\in[0,1],$ $p,q\geq 1$ and $s$ satisfies the following
relation:
$$\alpha=\left(\frac{1}{s}-\frac{1}{q}\right)\left(\frac{1}{n}-\frac{1}{p}+\frac{1}{s}\right)^{-1}.$$
\end{lemma}

The following lemma is a special case of Lemma \ref{GN} in two
dimension which will be useful in our context.

\begin{lemma}\label{LL}
For $\varphi\in C_0^{\infty}(\Theta)$, where $\Theta$ is the
admissible channel domain, we have the following estimate:
$$\|\varphi\|^4_{\mathrm{L}^4(\Theta)}\leq
2\|\varphi\|^2_{\mathrm{L}^2(\Theta)}\|\nabla
\varphi\|^2_{\mathrm{L}^2(\Theta)}.$$
\end{lemma}
The general proof in $\mathbb{R}^2$ of this lemma can be obtained
from Ladyzhenskaya \cite{La} (Chapter 1, Lemma 1).

\begin{remark}\label{LL1}
The above lemma suggests that $\V\cap \H\subset
\mathrm{L}^4(\Theta)$ and
\begin{align}\label{L1}
\mathrm{L}^2(0,T;\V)\cap \mathrm{L}^{\infty}(0,T;\H)\subset
\mathrm{L}^4((0,T)\times\Theta).\end{align}
\end{remark}

\begin{definition}
Let $X$ be a Banach space and let $X^{'}$ be its topological dual.
An operator $F:D\rightarrow X^{'},D\subset X$ is said to be monotone
if $$\left(F(x)-F(y),x-y\right)_{X^{'}\times X}\leq 0\textrm{ for
all } x,y\in D.$$
\end{definition}
$F$ is said to be $\Lambda$-monotone if $F+\Lambda I$ is monotone,
where $\Lambda\in\mathbb{R}$ and $I$ is the identity operator.

Next we prove the local monotonicity property.

\begin{lemma}[Local Monotonicity of $F$]\label{mon}
 Let us denote by $\mathbb{B}_{\varrho}$
the closed $\mathrm{L}^4$-ball in $\mathbf{V}$:
$$\mathbb{B}_{\varrho}=\left\{\x\in \mathbf{V}:\|\x\|_{\mathrm{L}^4}\leq
\varrho\right\}.$$  Define the nonlinear operator $F$ on
$\mathbf{V}$ by $F(\v)=-\nu A\v-B(\v,\v)-B_2(\w)\v-B_1(\w)\v$. Then
the operator $F+\Lambda I$, for $\Lambda\in\mathbb{R}$ and
sufficiently large $|\Lambda|$, is monotone in
$\mathbb{B}_{\varrho}$, i.e., for any $\v\in \mathbf{V}$ and $\x\in
\mathbb{B}_{\varrho},$
\begin{align}\label{31}
\left(F(\v)-F(\x),\z\right)+\Lambda|\z|^2\leq 0,
\end{align}
where $\z=\v-\x$ and
$\Lambda=\frac{\nu}{2C}-\frac{\beta^2_{20}(\omega)}{\nu}-|\eta_{21}(\omega)|-|\eta_{22}(\omega)|-\frac{27}{4}\frac{{\varrho}^4}{\nu^3}.$
$\beta_{20}$ is defined in (\ref{eqn1001a}) and $\eta_{2i}$'s are
defined in Remark \ref{rem11}.
\end{lemma}
\begin{proof}
We have $F(\v)=-\nu A\v-B(\v,\v)-B_2(\w)\v-B_1(\w)\v$ and
$F(\x)=-\nu A\x-B(\x,\x)-B_2(\w)\x-B_1(\w)\x$.

By a simple application of the property of the trilinear form, we
have
\begin{align*}
\left(B_2(\w)\v-B_2(\w)\x,\z\right)&=\left(B_2(\w)\v,\z\right)-\left(B_2(\w)\x,\z\right)\\
&=b(\v,\w,\z)-b(\x,\w,\z)\\
&=b(\v-\x,\w,\v-\x)=b(\z,\w,\z)\\
&=\left(B_2(\w)\z,\z\right).
\end{align*}
Similarly,
$$\left(B_1(\w)\v-B_1(\w)\x,\z\right)=\left(B_1(\w)\z,\z\right).$$
Now one can get
\begin{align}\label{32}
\left(F(\v)-F(\x),\z\right)=\left(\nu\Delta
\z,\z\right)-\left(B(\v,\v)-B(\x,\x),\z\right)-\left(B_2(\w)\z,\z\right)-\left(B_1(\w)\z,\z\right).\nonumber\\
\end{align}
Also we have $\left(\nu\Delta\z,\z\right)=-\nu\|\z\|^2$ and
\begin{align*}
\left(B(\v,\z),\x\right)&=b(\v,\z,\x)=-b(\v,\x,\z)=-b(\v,\x,\z)-b(\v,\z,\z)\\&=-b(\v,\x+\z,\z)=-b(\v,\v,\z)\\&=-\left(B(\v),\z\right).
\end{align*}
Similarly, one can prove that
$$\left(B(\v),\z\right)=-\left(B(\v,\z),\x\right)\textrm{ and
}\left(B(\x),\z\right)=-\left(B(\x,\z),\v\right).$$ This gives
\begin{align*}
\left(B(\v)-B(\x),\z\right)&=-\left(B(\v,\z),\x\right)+\left(B(\x,\z),\v\right)
=b(\x,\z,\v)-b(\v,\z,\x)\\
&=b(\x,\z,\v)-b(\x,\z,\x)+b(\x,\z,\x)-b(\v,\z,\x)\\
&=b(\x,\z,\v-\x)-b(\v-\x,\z,\x)=b(\x,\z,\z)-b(\z,\z,\x)\\&=-b(\z,\z,\x)=-\left(B(\z),\x\right).
\end{align*}
Using H\"{o}lder's inequality and Sobolev embedding theorem, one
gets
\begin{align*}
\left|\left(B(\v)-B(\x),\z\right)\right|&=\left|-\left(B(\z),\x\right)\right|=\left|b(\z,\z,\x)\right|\\
&\leq\|\z\|_{\mathrm{L}^4(\Theta)}\|\z\|_{\mathbf{V}}\|\x\|_{\mathrm{L}^4(\Theta)}\\
&\leq\|\z\|^{3/2}|\z|^{1/2}\|\x\|_{\mathrm{L}^4(\Theta)}\\
&\leq\frac{\nu}{4}\|\z\|^2+\frac{27}{4\nu^3}|\z|^2\|\x\|^4_{\mathrm{L}^4(\Theta)}.
\end{align*}
Also, we have $\left(B_1(\w)\z,\z\right)=b(\w,\z,\z)=0$. By using
the inequality $$\left|b(\u,\w,\v)\right|\leq
\beta_{20}|\u|^{1/2}\|\u\|^{1/2}|\v|^{1/2}\|\v\|^{1/2} +\sum_{k=1}^2
\beta_{2k}(t)|\u||\v|,$$  Young's inequality and Remark \ref{rem11},
we get
\begin{align*}
\left|\left(B_2(\w)\z,\z\right)\right|&=\left|b(\z,\w,\z)\right|\leq\beta_{20}(\omega)|\z|\|\z\|+\sum_{k=1}^2\beta_{2k}(t,\omega)|\z|^2\\
 &\leq\frac{\nu}{4}\|\z\|^2+\left(\frac{\beta^2_{20}(\omega)}{\nu}+|\eta_{21}(\omega)|+|\eta_{22}(\omega)|\right)|\z|^2.
\end{align*}
Applying all these estimates in (\ref{32}), one can deduce that
\begin{align*}
\left(F(\v)-F(\x),\z\right)\leq -\frac{\nu}{2}\|\z\|^2+\frac{27
{\varrho}^4}{4\nu^3}|\z|^2+\left(\frac{\beta^2_{20}(\omega)}{\nu}+|\eta_{21}(\omega)|+|\eta_{22}(\omega)|\right)|\z|^2.
\end{align*}
Since $\Theta$ is an admissible channel domain and Poincar\'{e}
inequality still holds, we get $|\z|\leq C\|\z\|$ and hence we have
\begin{align}\label{34}
\left(F(\v)-F(\x),\z\right)+\left(\frac{\nu}{2C}-\frac{\beta^2_{20}(\omega)}{\nu}-|\eta_{21}(\omega)|-|\eta_{22}(\omega)|-\frac{27
{\varrho}^4}{4\nu^3}\right)|\z|^2\leq 0.
\end{align}
\end{proof}

\begin{definition}[Strong Solution]
A strong solution $\mathbf{v}$ is defined on a given probability
space $(\Omega, \mathcal{F},\mathcal{F}_t,\mathbb{P})$ as a
$\mathrm{L}^2(\Omega;\mathrm{L}^2(0,T,\mathbf{V})\cap
C(0,T;\mathbf{H}))$ valued adapted process which satisfies the
stochastic channel flow model
\begin{align}\label{60}
\d\v+\left[B(\v,\v)+B_2(\w)\v+B_1(\w)\v\right]\d t&=\left[-\nu
A\v+P_{\mathbf{H}}\mathbf{f}_{\w}\right]\d t+g(t)\d W(t)\nonumber\\
\v(0)&=0,
\end{align}
in the weak sense and also the energy inequalities
\begin{align}\label{9a} &\mathbb{E}\left(\sup_{0\leq t\leq
T}\left|\v(\cdot,t)\right|^2+2\nu \int_0^T\|\v(\cdot,t)\|^2\d
s\right)\nonumber\\&\qquad\qquad\leq
C\left(T,\nu,\|\mathscr{F}\|_{_{\mathrm{L}^2(\Omega;H^1(0,T))}},\int_0^T
\mathrm{Tr}(g^*g)(t)\d t\right).
\end{align}
and
\begin{align}\label{36a}
&\mathbb{E}\left[\sup_{0\leq t\leq
T}\left|\v(\cdot,t)\right|^2e^{-\delta t}\right]+2\nu
\int_0^T\mathbb{E}\|\v(\cdot,t)\|^2e^{-\delta t}\d s
\nonumber\\&\quad\quad\leq
C\left(T,\nu,\|\mathscr{F}\|_{_{\mathrm{L}^2(\Omega;H^1(0,T))}},\delta,\int_0^T\Tr(g^*g)(t)e^{-\delta
t}\d t\right).
\end{align}
\end{definition}
\begin{definition}
Let $X$ be a Banach space and let $X'$ be its topological dual. An
operator $F:D\to X', D\subset X$ is said to be hemicontinuous at
$x\in D$, if $y\in X$, $t_n>0, n=1,2,\cdots$, $t_n\to 0$ and
$x+t_ny\in D$ imply $F(x+t_ny)\to F(x)$ weakly.
\end{definition}
\begin{theorem}\label{mainth}
Let $\mathbf{f}_{\w}\in \mathrm{L}^2(0,T;\mathbf{V}').$ Then there
exists a unique adapted process $\v(\x,t,\omega)$ with the
regularity
$$\v\in \mathrm{L}^2\left(\Omega;C(0,T;\mathbf{H})\cap
\mathrm{L}^2(0,T;\mathbf{V})\right)$$ satisfying the stochastic
model (\ref{60}) and the a-priori bounds in (\ref{9a}) and
(\ref{36a}).
\end{theorem}
\begin{proof}
\textbf{Part I: Existence of Strong Solution.}

We prove the existence of Strong solutions of the stochastic channel
flow model (\ref{60}) in the following four steps.

\noindent\textit{\textbf{Step (1)} Finite-dimensional (Galerkin)
approximation of the stochastic channel flow model (\ref{60}):}

Let $\{e_1, e_2,\cdots,e_n,\cdots\}$ be a complete orthonormal
system in $\H$ belonging to $\V$. Denote by $\H_n$ the
$n$-dimensional subspace of $\H$ generated with $\{e_1,
e_2,\cdots,e_n\}$. Let $\v^n$ be the solution of the following
stochastic differential equation in $\mathbf{H}_n$:
\begin{align}\label{main0}
\d(\v^n(t),v)&=(F(\v^n(t)),v)\d
t+(P_{\mathbf{H}}\mathbf{f}_{\w}(t),v)\d t+\sum_k(g_k(t),v)\d
W_k(t),\nonumber\\ \v^n(0)&=P_n\v(0)=0,
\end{align}
in $(0,T)$ for any $v\in\mathbf{H}_n$, where $F(\v)=-\nu
A\v-B(\v,\v)-B_1(\w)\v-B_2(\w)\v.$ Denoting
$\tilde{F}(\v)=F(\v)+P_{\mathbf{H}}\mathbf{f}_{\w}$, we have
$\{\v^n\}$ satisfies the stochastic It\^{o} differential equation
\begin{align}\label{61a}
\d \v^n(t)=\tilde{F}(\v^n(t))\d t+g(t)\d W(t),\;\; \v^n(0)=0
\end{align}
and the corresponding energy equality
\begin{align}\label{61ab}
\d |\v^n(t)|^2=2(\tilde{F}(\v^n(t)),\v^n(t))\d t+2\left(g(t)\d
W(t),\v^n(t)\right)+\textrm{Tr}(g^*g)(t)\d t.
\end{align}

\noindent\textit{\textbf{Step (2)}  Weak convergent sequences:}

Using the a-priori estimates in Theorem \ref{energy1} and Theorem
\ref{energy3}, it follows from the Banach-Alaoglu theorem that along
a subsequence, the Galerkin approximations $\{\v^n\}$ have the
following limits:
\begin{align}\label{61}
&\v^n\rightarrow\v\textrm{ weak star in }
\mathrm{L}^2(\Omega;\mathrm{L}^\infty (0,T;\mathbf{H})\cap
\mathrm{L}^2(0,T;\mathbf{V})) \nonumber\\& F(\v^n)\rightarrow
F_0\textrm{ weakly in }
\mathrm{L}^2(\Omega;\mathrm{L}^2(0,T;\mathbf{V}')).
\end{align}
Now the assertion in the second statement holds since $F(\v^n)$ is
bounded in the space
$\mathrm{L}^2(\Omega;\mathrm{L}^2(0,T;\mathbf{V}')).$ Note that
$\tilde{F}(\v^n(t))\to\tilde{F}_0(t)=F_0(t)+P_{\mathbf{H}}\mathbf{f}_{\w}(t)$
weakly in $\mathrm{L}^2(\Omega;\mathrm{L}^2(0,T;\mathbf{V}'))$ since
$\mathbf{f}_{\w}\in\mathrm{L}^2(\Omega;\mathrm{L}^2(0,T;\mathbf{V}'))$.

Moreover, $\v^n(t)\in\mathrm{L}^2(\Omega; C([0,T];\H_n))$ and the
Galerkin approximations $\v^n$ converge to $\v$ weakly star in the
Banach space $\mathrm{L}^2(\Omega; C([0,T];\H))$ implies that
$t\mapsto \v(t)$ is a continuous function from $[0,T]$ into $\H$
with probability $1$ (see page 46-47, Proposition 3.3 of Menaldi and
Sritharan \cite{MS}). Hence we have
$\v(t)\in\mathrm{L}^2(\Omega;C(0,T;\mathbf{H})\cap
\mathrm{L}^2(0,T;\mathbf{V}))$.

Now let us apply It\^{o}'s lemma to the function $e^{-r(t)}|x|^2$
and to the process $\v^n(t)$ to obtain
\begin{align}\label{64}
&\d\left[e^{-r(t)}|\v^n(t)|^2\right]=-e^{-r(t)}\left[\left(\dot{r}(t)\v^n(t),\v^n(t)\right)
+2\left(\tilde{F}(\v^n(t)),\v^n(t)\right)\right]\d
t\nonumber\\&\qquad\qquad\qquad\quad+2e^{-r(t)}\left(g(t)\d
W(t),\v^n(t)\right)+e^{-r(t)}\mathrm{Tr}(g^*g)(t)\d t,
\end{align}
where $\dot{r}(t)$ denotes the derivative of $r(t)$. Let us
integrate the above equality (\ref{64}) from 0 to $T$ and taking the
expectation to get
\begin{align*}
&\mathbb{E}\left[e^{-r(T)}|\v^n(T)|^2\right]=\mathbb{E}\left[\int_0^Te^{-r(t)}\left(2\tilde{F}(\v^n(t))-\dot{r}(t)\v^n(t),\v^n(t)\right)\d
t\right]\nonumber\\&+2\mathbb{E}\left[\int_0^Te^{-r(t)}(g(t)\d
W(t),\v^n(t))\right]+\mathbb{E}\left[\int_0^Te^{-r(t)}\mathrm{Tr}(g^*g)(t)\d
t\right].
\end{align*}
But $2\int_0^Te^{-r(t)}(g(t)\d W(t),\v^n(t))$ is a martingale having
expectation zero. Then from the above equation, we have
\begin{align}\label{65}
\mathbb{E}\left[e^{-r(T)}|\v^n(T)|^2\right]&=\mathbb{E}\left[\int_0^Te^{-r(t)}\left(2\tilde{F}(\v^n(t))-\dot{r}(t)\v^n(t),\v^n(t)\right)\d
t\right.\nonumber\\&\left.\quad+\int_0^Te^{-r(t)}\mathrm{Tr}(g^*g)(t)\d
t\right].
\end{align}
From (\ref{61}) and (\ref{61a}), we also note that $\v$ has the
It\^{o} differential
\begin{align}\label{62}
\d \v(t)=\tilde{F}_0(t)\d t+g(t)\d W(t)\textrm{ weakly in
}\mathrm{L}^2(\Omega;\mathrm{L}^2(0,T;\mathbf{V}')),
\end{align}
with $\v(0)=0$, where
$\tilde{F}_0(t)=F_0(t)+P_{\mathbf{H}}\mathbf{f}_{\w}(t)$. Also $\v$
satisfies the energy equality
\begin{align}\label{62ab}
\d|\v(t)|^2=2(\tilde{F}_0(t),\v(t))\d t+2\left(g(t)\d
W(t),\v(t)\right)+\textrm{Tr}(g^*g)(t)\d t.
\end{align}
By applying It\^{o}'s lemma to the function $e^{-r(t)}|x|^2$ and to
the process $\v(\x,t,\omega)$ from (\ref{62}), we can show that
$\v(\x,t,\omega)$ satisfies
\begin{align}\label{65a}
\mathbb{E}\left[e^{-r(T)}|\v(T)|^2\right]&=\mathbb{E}\left[\int_0^Te^{-r(t)}\left(2\tilde{F}_0(t)-\dot{r}(t)\v(t),\v(t)\right)\d
t\right.\nonumber\\&\left.\quad+\int_0^Te^{-r(t)}\mathrm{Tr}(g^*g)(t)\d
t\right].
\end{align}

\noindent\textit{\textbf{Step (3)}  Lower semicontinuity property:}

Since $\v^n\to \v$ weak star in
$\mathrm{L}^2(\Omega;\mathrm{L}^\infty (0,T;\mathbf{H})\cap
\mathrm{L}^2(0,T;\mathbf{V}))$, by the lower semicontinuity property
of $\mathrm{L}^2$-norm, we have
\begin{align}\label{65ab}
\liminf_{n}\mathbb{E}\left[e^{-r(T)}|\v^n(T)|^2\right]\geq\mathbb{E}\left[e^{-r(T)}|\v(T)|^2\right].
\end{align}
Let us take $\mathlarger{\liminf_{n}}$ on the equation (\ref{65})
and by using (\ref{65ab}) and (\ref{65a}) to obtain
\begin{align}\label{66}
&\liminf_{n}\mathbb{E}\left[\int_0^Te^{-r(t)}\left(2\tilde{F}(\v^n(t))-\dot{r}(t)\v^n(t),\v^n(t)\right)\d
t\right]\nonumber\\&=\liminf_{n}\mathbb{E}\left[e^{-r(T)}|\v^n(T)|^2-\int_0^Te^{-r(t)}\mathrm{Tr}(g^*g)(t)\d
t\right]\nonumber\\
&\geq\mathbb{E}\left[e^{-r(T)}|\v(T)|^2-\int_0^Te^{-r(t)}\mathrm{Tr}(g^*g)(t)\d t\right]\nonumber\\
&=\mathbb{E}\left[\int_0^Te^{-r(t)}\left(2\tilde{F}_0(t)-\dot{r}(t)\v(t),\v(t)\right)\d
t\right].
\end{align}
Hence, we get
\begin{align}\label{67}
&\liminf_{n}\mathbb{E}\left[\int_0^Te^{-r(t)}\left(2\tilde{F}(\v^n(t))-\dot{r}(t)\v^n(t),\v^n(t)\right)\d
t\right]\nonumber\\&\geq
\mathbb{E}\left[\int_0^Te^{-r(t)}\left(2\tilde{F}_0(t)-\dot{r}(t)\v(t),\v(t)\right)\d
t\right].
\end{align}

\noindent\textit{\textbf{Step (4)}  Local Minty-Browder technique:}

For $\y(\omega,\x,t)\in\mathrm{L}^2(\Omega;\mathrm{L}^{\infty}
(0,T;\mathbf{H}_m))$ with $m\leq n$, let us set
\begin{align}\label{63}
r(t)=r(t,\omega)=2C(\omega)t+\frac{27}{2\nu
^3}\int_0^t\|\y(s)\|_{\mathrm{L}^4}^4\d s,
\end{align}
where
$C(\omega)=\frac{\beta_{20}^2(\omega)}{\nu}+|\eta_{21}(\omega)|+|\eta_{22}(\omega)|-\frac{\nu}{2C}$
as an adapted, continuous (and bounded in $\omega$) real-valued
process in $[0,T]$. Note that from Lemma \ref{LL} and Remark
\ref{LL1}, $r(t,\omega)$ is well defined.

Since
$\y(t)=\y(\omega,\x,t)\in\mathrm{L}^2(\Omega;\mathrm{L}^{\infty}
(0,T;\mathbf{H}_m)$ belongs to the closed $\mathrm{L}^4$-ball in
$\V$, from the Local Monotonicity Lemma \ref{mon}, by setting
$\z(t)=\v^n(t)-\y(t)$ and
$\tilde{F}(\v^n)=F(\v^n)+P_{\mathbf{H}}\mathbf{f}_{\w}$ with
$\mathbf{f}_{\w}\in\mathrm{L}^2(0,T;\V')$, we have
\begin{align}\label{68}
&\mathbb{E}\left[\int_0^Te^{-r(t)}\left(2\tilde{F}(\v^n(t))-2\tilde{F}(\y(t)),\v^n(t)-\y(t)\right)\d
t\right]\nonumber\\&-\mathbb{E}\left[\int_0^Te^{-r(t)}\dot{r}(t)|\v^n(t)-\y(t)|^2\d
t\right]\leq 0.
\end{align}
On rearranging the terms in the inequality (\ref{68}), we get
\begin{align}\label{68a}
&\mathbb{E}\left[\int_0^Te^{-r(t)}\left(2\tilde{F}(\v^n(t))-\dot{r}(t)\v^n(t),\v^n(t)-\y(t)\right)\d
t\right]\nonumber\\&\qquad\leq
\mathbb{E}\left[\int_0^Te^{-r(t)}\left(2\tilde{F}(\y(t))-\dot{r}(t)\y(t),\v^n(t)-\y(t)\right)\d
t\right].
\end{align}
Taking the limit as $n\rightarrow\infty$ in (\ref{68a}), one obtains
\begin{align}\label{68b}
&\liminf_n\mathbb{E}\left[\int_0^Te^{-r(t)}\left(2\tilde{F}(\v^n(t))-\dot{r}(t)\v^n(t),\v^n(t)-\y(t)\right)\d
t\right]\nonumber\\&\qquad\leq
\mathbb{E}\left[\int_0^Te^{-r(t)}\left(2\tilde{F}(\y(t))-\dot{r}(t)\y(t),\v(t)-\y(t)\right)\d
t\right].
\end{align}
By using (\ref{67}) and (\ref{68b}), we get
\begin{align}\label{67ab}
&
\mathbb{E}\left[\int_0^Te^{-r(t)}\left(2\tilde{F}_0(t)-\dot{r}(t)\v(t),\v(t)-\y(t)\right)\d
t\right]\nonumber\\&\leq
\liminf_{n}\mathbb{E}\left[\int_0^Te^{-r(t)}\left(2\tilde{F}(\v^n(t))-\dot{r}(t)\v^n(t),\v^n(t)-\y(t)\right)\d
t\right]\nonumber\\&\leq
\mathbb{E}\left[\int_0^Te^{-r(t)}\left(2\tilde{F}(\y(t))-\dot{r}(t)\y(t),\v(t)-\y(t)\right)\d
t\right].
\end{align} From the above inequality, we get  \begin{align}\label{67abc}
&
\mathbb{E}\left[\int_0^Te^{-r(t)}\left(2\tilde{F}_0(t)-\dot{r}(t)\v(t),\v(t)-\y(t)\right)\d
t\right]\nonumber\\&\leq
\mathbb{E}\left[\int_0^Te^{-r(t)}\left(2\tilde{F}(\y(t))-\dot{r}(t)\y(t),\v(t)-\y(t)\right)\d
t\right].
\end{align} On rearranging the terms in the inequality (\ref{67abc}), we obtain
\begin{align}\label{eq007}
&\mathbb{E}\left[\int_0^Te^{-r(t)}\left(2\tilde{F}_0(t)-2\tilde{F}(\y(t)),\v(t)-\y(t)\right)\d
t\right]\nonumber\\&-\mathbb{E}\left[\int_0^Te^{-r(t)}\dot{r}(t)|\v(t)-\y(t)|^2\d
t\right]\leq 0.
\end{align}
This estimate holds for any
$\y\in\mathrm{L}^2(\Omega;\mathrm{L}^{\infty}(0,T;\mathbf{H}_m))$
for any $m\in\mathbb{N}$.  It is clear by a density argument that
the above inequality remains the same for any $\y\in
\mathrm{L}^2(\Omega;C(0,T;\mathbf{H})\cap
 \mathrm{L}^2(0,T;\mathbf{V}))$.  Indeed, for
any $\y\in\mathrm{L}^2(\Omega;C(0,T;\mathbf{H})\cap
 \mathrm{L}^2(0,T;\mathbf{V})),$ there exits a
strongly convergent sequence $\y_m\in
\mathrm{L}^2(\Omega;C(0,T;\mathbf{H})\cap
\mathrm{L}^2(0,T;\mathbf{V}))$ that satisfies inequality
(\ref{eq007}).

Let us now take for any $\lambda>0$, $\y(t)=\v(t)-\lambda \z(t)$ in
(\ref{eq007}). Note that
$\v\in\mathrm{L}^2(\Omega;C(0,T;\mathbf{H})\cap
\mathrm{L}^2(0,T;\mathbf{V}))$ satisfies the It\^{o} differential
equation given by (\ref{62}) and $\z$ is an adapted process in
$\mathrm{L}^2(\Omega;C(0,T;\mathbf{H})\cap
\mathrm{L}^2(0,T;\mathbf{V})).$ Then from (\ref{eq007}), we have
\begin{align*}
\lambda
\mathbb{E}\left[\int_0^Te^{-r(t)}\left(2\tilde{F}_0(t)-2\tilde{F}(\v(t)-\lambda
\z(t)),\z(t)\right)\d t-\lambda
\int_0^Te^{-r(t)}\dot{r}(t)|\z(t)|^2\d t\right]\leq 0.
\end{align*}
Dividing by $\lambda$ on both sides of the above inequality and
letting $\lambda$ go to $0$ and using the hemicontinuity property of
$F$, one obtains
\begin{align*}
\mathbb{E}\left[\int_0^Te^{-r(t)}\left(\tilde{F}_0(t)-\tilde{F}(\v(t)),\z(t)\right)\d
t\right]\leq 0.
\end{align*}
(For more details see Fernando, Sritharan and Xu \cite{FSX}.) Since
$\z$ is arbitrary, we conclude that
$\tilde{F}_0(t)=\tilde{F}(\v(t))$. Thus the existence of the strong
solutions of the stochastic system (\ref{60}) has been proved.\\

\noindent \textbf{Part II: Pathwise Uniqueness of Strong Solution.} \\
Let $\v_1\in \mathrm{L}^2(\Omega;C(0,T;\mathbf{H})\cap
\mathrm{L}^2(0,T;\mathbf{V}))$ be another solution of the equation
(\ref{60}). Then $\vartheta=\v-\v_1$ solves the stochastic
differential equation in
$\mathrm{L}^2(\Omega;\mathrm{L}^2(0,T,\mathbf{V}'))$, given by
\begin{align}\label{69}
\d \vartheta=\left({F}(\v(t))-{F}(\v_1(t))\right)\d t,\;\;
\vartheta(0)=0.
\end{align}
On multiplying (\ref{69}) by $e^{-r(t)}\vartheta(t)$, we obtain
\begin{align*}
\d
\left[e^{-r(t)}|\vartheta(t)|^2\right]=-e^{-r(t)}\dot{r}(t)|\vartheta(t)|^2\d
t+2e^{-r(t)}\left(F(\v(t))-F(\v_1(t)),\vartheta(t)\right)\d t.
\end{align*}
Now let us use the local monotonicity condition (\ref{31}). Remark
\ref{LL1} gives $\v_1$ in the $\mathrm{L}^4$-ball of $\V$, then by
taking $\z=\v-\v_1$ in (\ref{34}) and by choosing $r(t)$ as in
(\ref{63}), one gets
\begin{align*}
\d\left[e^{-r(t)}|\vartheta(t)|^2\right]\leq 0\Rightarrow
e^{-r(t)}|\vartheta(t)|^2\leq |\vartheta(0)|^2=0.
\end{align*}
Hence the uniqueness of $\v(\x,t,\omega)$ satisfying the system of
equations (\ref{6}) - (\ref{7}) has been proved.
\end{proof}

\section{Characterization of the Perturbation Pressure}
In this section, we characterize the perturbation pressure,
$q(\cdot,\cdot,\cdot)=p(\cdot,\cdot,\cdot)-P(\cdot,\cdot,\cdot)$ by
using a generalization of de Rham's theorem (\cite{SJ}) to processes
(see \cite{LRS} for more details). We have constructed the basic
vector field in such a way that the random basic pressure field
$P(\x,t,\omega)$ tends to infinity in each of the outlets
$\mathbf{O}_i$ as $|x|\to\infty$.

\begin{theorem}\label{pt1}
[A Generalization of de Rham's theorem to Processes] Let $D$ be a
bounded, connected and Lipschitz open subset of $\mathbb{R}^2$.
Given $r_0\in[1,\infty]$, $r_1\in[1,\infty]$ and $s_1\in\mathbb{Z},$
let
\begin{equation}\label{eqpt1}
\chi\in
\mathrm{L}^{r_0}\left(\Omega;\mathrm{W}^{s_1,r_1}(0,T;H^{-1}(D))\right)
\end{equation}
be such that, for all $v\in\mathscr{D}(D),$ $\nabla\cdot v=0$,
$\mathbb{P}$ - a. s.,
\begin{equation}\label{eqpt2}
\left<\chi,v\right>_{\mathscr{D}'(D)\times\mathscr{D}(D)}=0\textrm{
in } \mathscr{D}'(0,T),
\end{equation}
where $\mathscr{D}'(D)$ denotes the dual space of $\mathscr{D}(D)$.
Then, there exists a unique
\begin{equation}\label{eqpt3}
q\in\mathrm{L}^{r_0}\left(\Omega;
\mathrm{W}^{s_1,r_1}(0,T;\mathrm{L}^2(D))\right)
\end{equation}
such that $\mathbb{P}$ - a. s.,
\begin{align}\label{eqpt4}
\nabla q&=\chi\textrm{ in } \mathscr{D}'((0,T)\times D),\\
\int_Dq\d \mathbf{x}&=0\textrm{ in }\mathscr{D}'(0,T).
\end{align}
\end{theorem}
\begin{proof}
For proof see Theorem 4.1 of \cite{LRS}.
\end{proof}
\begin{lemma}\label{pt1b}
Let $\Theta$ be an admissible channel domain, let
$\tilde{\Theta}\subset\Theta$ ($\tilde{\Theta}$ is non-empty) be a
bounded subdomain with $\overline{\tilde{\Theta}}\subset\Theta$. For
any given $r_0\in[1,\infty]$, $r_1\in[1,\infty]$ and
$s_1\in\mathbb{Z}$, let
\begin{equation}\label{eqpt1a} \chi\in
\mathrm{L}^{r_0}\left(\Omega;\mathrm{W}^{s_1,r_1}(0,T;H^{-1}(\Theta))\right)
\end{equation}
be such that, for all $v\in\mathscr{D}(\Theta),$ $\nabla\cdot v=0$,
$\mathbb{P}$ - a. s.,
\begin{equation}\label{eqpt2a}
\left<\chi,v\right>_{\mathscr{D}'(\Theta)\times\mathscr{D}(\Theta)}=0\textrm{
in } \mathscr{D}'(0,T).
\end{equation} Then there exists a unique
$$q\in\mathrm{L}^{r_0}\left(\Omega;
\mathrm{W}^{s_1,r_1}(0,T;\mathrm{L}^2_{loc}(\Theta))\right),$$
satisfying $$\chi=\nabla q$$ in the sense of distributions in
$(0,T)\times\Theta$, and $$\int_{\tilde{\Theta}}q\d \x=0$$ for
almost all $t\in[0,T]$.
\end{lemma}
The proof for deterministic case can be found in Lemma 1.4.2,
Chapter $2$, page 202 of Sohr \cite{SoH} for any general unbounded
domain in $\mathbb{R}^n,n\geq 2$. Also see the Remark 4.3 of  Langa,
Real, and Simon \cite{LRS} for stochastic case.

We will now prove the existence and uniqueness of the perturbation
pressure using Lemma \ref{pt1b}.

\begin{theorem}
There exists a unique scalar distribution
$$q\in\mathrm{L}^1\left(\Omega;\mathrm{W}^{-1,\infty}(0,T;\mathrm{L}^2_{loc}(\Theta))\right)$$
such that (\ref{6})-(\ref{7}) are satisfied in the sense of
distributions.
\end{theorem}
\begin{proof}
By the existence and uniqueness theorem on strong solutions, there
exists a process $\v(\x,t,\omega)$ satisfying (\ref{6})-(\ref{7})
and
\begin{equation}\label{eqpt6}
\v(\x,t,\omega)\in\mathrm{L}^2\left(\Omega;C(0,T;\mathbf{H})\cap
\mathrm{L}^2(0,T;\mathbf{V})\right).
\end{equation}
Also $\v(\x,t,\omega)$ satisfies the variational equation
$\mathbb{P}$ - a. s., for all $t\in[0,T]$
\begin{align}\label{eqpt7}
&\int_{\Theta}\v(t)\cdot
\mathcal{Z}\d\x=-\nu\int_0^t\int_{\Theta}\nabla\v(s)\cdot\nabla
\mathcal{Z}\d\x\d
s-\int_0^t\int_{\Theta}(\v(s)\cdot\nabla\v(s))\cdot
\mathcal{Z}\d\x\d
s\nonumber\\
&\qquad-\int_0^t\int_{\Theta}(\v(s)\cdot\nabla\w(s))\cdot
\mathcal{Z}\d\x\d
s-\int_0^t\int_{\Theta}(\w(s)\cdot\nabla\v(s))\cdot
\mathcal{Z}\d\x\d
s\nonumber\\&\qquad+\int_0^t\left<\mathbf{f}_{\w}(s),\mathcal{Z}\right>_{H^{-1}(\Theta)\times
H_0^1(\Theta)}\d s+\int_0^t(g(s)\d W(s), \mathcal{Z}),
\end{align}
for all $\mathcal{Z}\in H_0^1(\Theta)$ such that
$\nabla\cdot\mathcal{Z}=0$. Let us differentiate (\ref{eqpt7}) with
respect to $t$ ($\omega\in\Omega$ being fixed), we get in
$\mathscr{D}'(0,T)$
\begin{align}\label{eqpt8}
&-\int_{\Theta}\partial_t\v\cdot
\mathcal{Z}\d\x-\nu\int_{\Theta}\nabla\v\cdot\nabla
\mathcal{Z}\d\x-\int_{\Theta}(\v\cdot\nabla\v)\cdot
\mathcal{Z}\d\x-\int_{\Theta}(\w\cdot\nabla\v)\cdot
\mathcal{Z}\d\x\nonumber\\&-\int_{\Theta}(\v\cdot\nabla\w)\cdot
\mathcal{Z}\d\x+\left<\mathbf{f}_{\w}(\cdot),\mathcal{Z}\right>_{H^{-1}(\Theta)\times
H_0^1(\Theta)}+(\dot{\mathscr{G}}(\cdot), \mathcal{Z})=0.
\end{align}

Since $\mathcal{Z}\in H_0^1(\Theta),$ and $\left(\nabla\v,\nabla
\mathcal{Z}\right)=-\left(\Delta\v,\mathcal{Z}\right),$ from the
equation (\ref{eqpt8}), we have
\begin{align}\label{eqpt9}
\left<-\partial_t\v+\nu\Delta\v-\v\cdot\nabla\v-\v\cdot\nabla\w-\w\cdot\nabla\v+\mathbf{f}_{\w}+\dot{\mathscr{G}}(\x,t),\mathcal{Z}\right>_{_{H^{-1}(\Theta)\times
H_0^1(\Theta)}}=0.
\end{align}
Let us denote
$\chi=-\partial_t\v+\nu\Delta\v-\v\cdot\nabla\v-\v\cdot\nabla\w-\w\cdot\nabla\v+\mathbf{f}_{\w}+\dot{\mathscr{G}}(\x,t)$.
We will prove that
\begin{equation}\label{eqpt10}
\chi\in\mathrm{L}^1\left(\Omega;
\mathrm{W}^{-1,\infty}(0,T;H^{-1}(\Theta))\right).
\end{equation}
Since $\partial_t$ is linear continuous from
$\mathrm{L}^{\infty}(0,T;\mathbf{H})$ into
$\mathrm{W}^{-1,\infty}(0,T;\mathbf{H})$ and then into
$\mathrm{W}^{-1,\infty}(0,T;H^{-1}(\Theta))$, by using
(\ref{eqpt6}), we have $$\partial_t\v\in\mathrm{L}^1\left(\Omega;
\mathrm{W}^{-1,\infty}(0,T;H^{-1}(\Theta))\right).$$ We know that
$\Delta$ is a linear continuous operator from $H^1(\Theta)$, and
then from $\mathbf{V}$, into $H^{-1}(\Theta)$. Hence (\ref{eqpt6})
implies that
$$\nu\Delta\v\in\mathrm{L}^2\left(\Omega;\mathrm{L}^2(0,T,H^{-1}(\Theta))\right)\subset\mathrm{L}^1\left(\Omega;\mathrm{L}^1(0,T,H^{-1}(\Theta))\right).$$
Now for every $\mathbf{v}\in\mathrm{L}^1(0,T;H^{-1}(\Theta))$, we
have $\mathbf{v}=\partial_t\int_0^{\cdot}\mathbf{v}$ and
$\int_0^{\cdot}\mathbf{v}\in
\mathrm{L}^{\infty}(0,T;H^{-1}(\Theta))$. This gives $\mathbf{v}\in
W^{-1,\infty}(0,T;H^{-1}(\Theta)),$ by using the definition of
$W^{-1,\infty}(0,T;H^{-1}(\Theta))$. Hence, we have the topological
embedding \begin{equation}\label{eqpt11}
\textrm{L}^1(0,T;H^{-1}(\Theta))\subset
\textrm{W}^{-1,\infty}(0,T;H^{-1}(\Theta)).
\end{equation}
Hence, we get $$\nu\Delta \v\in
\mathrm{L}^1\left(\Omega;W^{-1,\infty}(0,T;H^{-1}(\Theta))\right).$$

From Lemma \ref{buu}, Lemma \ref{bwv}, Lemma \ref{bvw1}, Lemma
\ref{bvw}, it is straight forward that the terms $\v\cdot\nabla \v$,
$\v\cdot\nabla\w$ and $\w\cdot\nabla \v$ are in
$\mathrm{L}^2\left(\Omega;\mathrm{L}^2(0,T;H^{-1}(\Theta))\right)$
which is clearly included in the space $\mathrm{L}^1\left(\Omega;
\mathrm{W}^{-1,\infty}(0,T;H^{-1}(\Theta))\right)$.

The term $\mathbf{f}_{\w}\equiv 0$ in $\mathbf{O}_i\backslash
\hat{\Theta}_0$ and from the estimate
$\w_t-\nu\Delta\w+\w\cdot\nabla\w\in
\mathrm{L}^2\left(\Omega;\mathrm{L}^2(0,T;\mathrm{L}^2(\Theta_0))\right)\subset\mathrm{L}^2\left(\Omega;\mathrm{L}^2(0,T;H^{-1}(\Theta_0))\right),$
we have $$\mathbf{f}_{\w}\in
\mathrm{L}^2\left(\Omega;\mathrm{L}^2(0,T;H^{-1}(\Theta))\right)\subset\mathrm{L}^1\left(\Omega;\mathrm{L}^1(0,T;H^{-1}(\Theta))\right).$$
Hence, by the embedding (\ref{eqpt11}), we get
$\mathbf{f}_{\w}\in\mathrm{L}^1\left(\Omega;
\mathrm{W}^{-1,\infty}(0,T;H^{-1}(\Theta))\right)$.

Since $\int_0^{\cdot}g(s)\d
W(s)\in\mathrm{L}^2(\Omega;C([0,T];\mathrm{L}^2(\Theta))),$ its
time-derivative $\dot{\mathscr{G}}(\cdot,\cdot,\cdot)$ satisfies
$\dot{\mathscr{G}}(\x,t,\omega)\in\mathrm{L}^2\left(\Omega;\mathrm{W}^{-1,\infty}(0,T;\mathrm{L}^2(\Theta))\right),$
which is clearly included in the required space. Hence, we have
(\ref{eqpt10}).

Finally by applying Lemma \ref{pt1b}, we get a unique
$$q\in\mathrm{L}^1\left(\Omega;
\mathrm{W}^{-1,\infty}(0,T;\mathrm{L}^2_{loc}(\Theta))\right)$$ such
that $\chi=\nabla q$ and the Navier Stokes Equations are satisfied
in the distributional sense.
\end{proof}

\section{Conclusion}\setcounter{equation}{0}

\begin{theorem}\label{thm11}
Let $\Theta$ be an admissible channel domain. Then for any given
random net flux
$\mathscr{F}(t,\omega)\in\mathrm{L}^2(\Omega;H^1(0,T))$ through the
channel domain, there exits a unique pathwise strong solution
$\u(\x,t,\omega)$ and pressure $p(\x,t,\omega)$ such that,
$$\u-\w\in \mathrm{L}^2\left(\Omega;C(0,T;\mathbf{H})\cap
\mathrm{L}^2(0,T;\mathbf{V})\cap\mathrm{L}^{\infty}(0,T:\mathbf{H})\right)$$
and
$$p-P\in\mathrm{L}^1\left(\Omega;\mathrm{W}^{-1,\infty}(0,T;\mathrm{L}^2_{loc}(\Theta))\right)$$
of the problem (\ref{1})-(\ref{5}), where $\w$ is the constructed
basic vector field with $\int_{\Gamma}\u\cdot\n\d
S=\int_{\Gamma}\w\cdot\n\d S=\mathscr{F}(t,\omega)$ for each
$\omega\in\Omega$ and $P$ is the constructed scalar field.
\end{theorem}
\begin{proof}
The existence of the solution $\u(\x,t,\omega)$, satisfying the
prescribed flux condition $\mathscr{F}(t,\omega)$ is given by the
existence of the basic vector field $\w(\x,t,\omega)$, which carries
the flux $\mathscr{F}(t,\omega)$ and the divergence free, zero flux
vector field $\v(\x,t,\omega)$.

For the pathwise uniqueness of the solutions, let us assume that
there exists two solutions $\u_1(\x,t,\omega)$ and
$\u_2(\x,t,\omega)$ satisfying the prescribed flux condition
$\mathscr{F}(t,\omega)$. Hence the difference field
$\mathfrak{w}(t)=\mathfrak{w}(\x,t,\omega)=\u_1(\x,t,\omega)-\u_2(\x,t,\omega)$
satisfies the equation
\begin{align}
&\d\mathfrak{w}(t)=-\nu A\mathfrak{w}(t)\d
t-[B(\mathfrak{w}(t),\mathfrak{w}(t))+B(\u_1(t),\mathfrak{w}(t))+B(\mathfrak{w}(t),\u_2(t))]\d
t,\nonumber\\
& \mathfrak{w}(0)=0,\nonumber\\
&\int_{\Gamma}\mathfrak{w}\cdot\n\d S=0,\;\forall\;\omega\in\Omega,
\end{align}
with zero flux condition and which is similar to the system of
equations in (\ref{eqn111}) without external noise term . The
uniqueness of the solutions follows from the uniqueness of this
system, the properties of linear and bilinear operators and by using
the Poincar\'{e} inequality for the admissible channel domain.

The uniqueness of the pressure field $p(\x,t,\omega)$ follows from
the uniqueness of the constructed basic scalar field
$P(\x,t,\omega)$ in section 3 and the uniqueness of the perturbed
pressure $q(\x,t,\omega)$.
\end{proof}

\begin{remark}
Theorem \ref{thm11} can be extended to the case, where the original
problem (\ref{1})-(\ref{5}) has a multiplicative Gaussian noise or a
multiplicative L\'{e}vy noise as external forcing. For this case the
basic vector field can be constructed in the same way as discussed
in this paper. The method of proof of the existence and uniqueness
of the perturbed vector field, under the suitable assumptions
(growth property and Lipschitz's condition) on the noise
coefficients, can be obtained from \cite{Ma4}.
\end{remark}

\medskip\noindent
{\bf Acknowledgements:} Manil T. Mohan would like to thank Council
of Scientific and Industrial Research (CSIR) for a Senior Research
Fellowship. Utpal Manna's work has been supported by National Board
of Higher Mathematics (NBHM). S. S. Sritharan's work has been funded
by U. S. Army Research Office, Probability and Statistics program.
The authors would also like to thank the reviewer for his critical
and valuable comments. Utpal Manna and Manil T. Mohan would like to
thank Indian Institute of Science Education and Research(IISER) -
Thiruvananthapuram for providing stimulating scientific environment
and resources.

\end{document}